\documentclass[11pt,a4paper]{amsart}

\usepackage{amsthm}
\usepackage{comment}

\usepackage{lipsum}
\usepackage{amsfonts}
\usepackage{framed}
\usepackage{caption,subcaption,graphicx} 
\usepackage{epstopdf}
\usepackage{amsmath,amssymb}
\usepackage{hyperref}
\usepackage[noabbrev, capitalise, nameinlink]{cleveref}
\usepackage{multirow}
\usepackage{mathtools}
\usepackage{tikz,pgfplots}
\usepackage{pgfplotstable}
\usepackage[square,numbers]{natbib}
\usepackage{easytable}
\usepackage{tabularray}
\usepackage{dsfont}
\usepackage{algorithm}
\usepackage{algpseudocode}
\usepackage{graphbox}

\pgfplotsset{
    compat=1.18,
    ErrorPlot/.style={
        p2/.style={semithick, mark options={solid}, mark=Mercedes star, mark size=2},
        p3/.style={semithick, mark options={solid}, mark=oplus, mark size=2},
        p4/.style={semithick, mark options={solid}, mark=square, mark size=2},
        p5/.style={semithick, mark options={solid}, mark=diamond, mark size=2},
        p6/.style={semithick, mark options={solid}, mark=triangle, mark size=2},
        p7/.style={semithick, mark options={solid}, mark=o, mark size=2},
		mstar/.style={semithick, mark options={solid}, mark=10-pointed star, mark size=2},
        fem/.style={semithick, blue},
        order/.style={semithick, gray, dashed}
    },
    width=.4\textwidth,
    legend style={font=\small, at={(0.65, -0.3)},anchor=north west,legend columns=3,draw=darkgray176,},
    log basis x={2},
    log basis y={10},
    tick align=outside,
    tick pos=left,
    x grid style={darkgray176},
    xlabel={H},
    xtick style={color=black},
    y grid style={darkgray176},
    ytick style={color=black},
    minor tick style={draw=none}
}

\definecolor{darkgray176}{RGB}{176,176,176}
\definecolor{gray}{RGB}{128,128,128}

\usepackage{color}
\usepackage{mathtools}
\usepackage[]{todonotes}

\newtheorem{theorem}{Theorem}

\newtheorem{lemma}[theorem]{Lemma}
\theoremstyle{definition}

\theoremstyle{remark}
\newtheorem{remark}[theorem]{Remark}

\Crefname{assumption}{Assumption}{Assumptions}
\numberwithin{theorem}{section}
\numberwithin{equation}{section}
\numberwithin{table}{section}
\numberwithin{figure}{section}

\ifpdf
\DeclareGraphicsExtensions{.eps,.pdf,.png,.jpg}
\else
\DeclareGraphicsExtensions{.eps}
\fi

\newcommand{\ub}{\boldsymbol{u}}
\newcommand{\vb}{\boldsymbol{v}}
\newcommand{\wb}{\boldsymbol{w}}
\newcommand{\fb}{\boldsymbol{f}}
\newcommand{\eb}{\boldsymbol{e}}
\newcommand{\gb}{\boldsymbol{g}}
\newcommand{\Vh}{\mathbf{V}}
\newcommand{\Rd}{\mathbb{R}^d}

\newcommand{\dx}{\,\mathrm{d}x}

\makeatletter
\def\author@andify{%
  \nxandlist {\unskip ,\penalty-1 \space\ignorespaces}%
    {\unskip {} \@@and~}%
    {\unskip \penalty-2 \space \@@and~}%
}
\renewcommand{\andify}{%
  \nxandlist{\unskip, }{\unskip{} \@@and~}{\unskip \space \@@and~}}
\makeatother

\title[SLOD for Heterogeneous Linear Elasticity]{Super-Localized Orthogonal Decomposition Method\\for Heterogeneous Linear Elasticity}

\author[C. Belponer]{Camilla Belponer}
\author[J. C. Garay]{Jos\'e C. Garay}
\author[P. Munch]{Peter Munch}
\author[D. Peterseim]{Daniel Peterseim}

\address[C. Belponer, J. C. Garay, D. Peterseim]{Institute of Mathematics, University of Augsburg, Universit\"atsstr.~12a, 86159 Augsburg, Germany}
\email{camilla.belponer@uni-a.de, jose.garay.fernandez@uni-a.de, daniel.peterseim@uni-a.de}

\address[P. Munch]{Institute of Mathematics, Technische Universit\"at Berlin, Strasse des 17. Juni 136, 10623 Berlin Gemany}
\email{muench@math.tu-berlin.de}

\address[D. Peterseim]{Centre for Advanced Analytics and Predictive Sciences (CAAPS), University of Augsburg, Universit\"atsstr.~12a, 86159 Augsburg, Germany}

\thanks{
The research of C. Belponer and D. Peterseim has been funded by
the Deutsche Forschungsgemeinschaft (DFG, German Research Foundation), grant DFG CA 1159/1-4 and PE 2143/1-6 
``Computational multiscale methods for inverse estimation of effective properties of poroelastic tissues''. \\
The work of J.C.~Garay is part of a project that has received funding from the European Research Council (ERC) under the European Union's Horizon 2020 research and innovation programme (Grant agreement No.~865751 -- RandomMultiScales).}

\begin{document}

\begin{abstract}
We present the Super-Localized Orthogonal Decomposition (SLOD) method for the numerical homogenization of linear elasticity problems with multiscale microstructures modeled by a heterogeneous coefficient field without any periodicity or scale separation assumptions. Compared to the established Localized Orthogonal Decomposition (LOD) and its linear localization approach, SLOD achieves significantly improved sparsity properties through a nonlinear superlocalization technique, leading to computationally efficient solutions with significantly less oversampling -- without compromising accuracy. We generalize the method to vector-valued problems and provide a supporting numerical analysis. We also present a scalable implementation of SLOD using the \texttt{deal.II} finite element library, demonstrating its feasibility for high-performance simulations. Numerical experiments illustrate the efficiency and accuracy of SLOD in addressing key computational challenges in multiscale elasticity.
\end{abstract}

\maketitle


\section{Introduction}
The partial differential equations (PDEs) of linear elasticity are fundamental tools in structural analysis, describing how solid objects deform and experience stress under small loads. In this work, we consider solid materials endowed with highly heterogeneous microstructures encoded in spatially varying PDE coefficients. We specifically avoid assumptions of periodicity or scale separation, which can significantly restrict the range of applicability of classical homogenization theories.

When the microstructure of the material exhibits multiple characteristic length scales, the finite element discretization of the resulting multiscale PDE often requires extremely fine meshes to capture all relevant features. Such meshes lead to prohibitively large systems of equations — especially in elasticity, where the displacement field is vector-valued and the computational burden is further amplified.

Numerical homogenization methods address these challenges by designing \emph{generalized} finite element spaces whose basis functions capture essential microscopic behavior while operating at a coarse, macroscopic scale. A variety of such methods have been proposed, including the Localized Orthogonal Decomposition (LOD) method~\cite{MaP14} (see also~\cite{KPY18, PeS16, HeP13, BrennerLOD}), the Multiscale Finite Element Method (MsFEM)~\cite{HoW97,EH09}, the Generalized Multiscale Finite Element Method (GMsFEM)~\cite{EFENDIEV2013116}, the multiscale Generalized Finite Element Method (MS-GFEM)~\cite{BaL11, BLS20, Ma21, ma2024unifiedframeworkmultiscalespectral}, Bayesian approaches~\cite{Owh15}, rough polyharmonic splines~\cite{OZB14}, Gamblets~\cite{Owh17}, and multiscale methods inspired by FETI-DP and BDDC frameworks~\cite{MaSa21, KKR18}. General overviews can be found in~\cite{OwhS19, MalP20, CEH23, blanc2023homogenization} and in the review~\cite{AHP21}.

Among these techniques, the LOD approach has proven successful both theoretically and empirically for a range of PDEs beyond the prototypical Poisson problem, including heterogeneous linear elasticity~\cite{HeP16} and more general multiphysics problems such as thermoelasticity~\cite{MaPersson17} and poroelasticity~\cite{FuACMPP19,AltCMPP20}. In LOD, the degrees of freedom are associated with the elements (or other suitable entities such as vertices or edges) of a coarse mesh, chosen independently of the underlying fine-scale heterogeneities. Microscopic details are then incorporated through local patch problems---often called \emph{cell problems}---in slightly enlarged regions around these coarse elements (oversampling patches). The optimal convergence rates of the error with respect to the coarse mesh size $H$ typically require an oversampling region of diameter $\mathcal{O}(H \,\lvert \log H\rvert)$. Although the appearance of a logarithmic factor is optimal in existing constructions and common among numerical homogenization methods, the associated moderately increased communication and decreased sparsity still pose a bottleneck in large-scale applications, where the vector-valued nature of the problem further amplifies the computational cost.

An improved localization strategy, named \emph{Super-Localized Orthogonal Decomposition}(SLOD), was introduced in~\cite{HaPe21b} for a prototype scalar diffusion problem to address precisely this issue. By relaxing the linear structure of the classical LOD construction, SLOD achieves (in practice) \emph{superexponential} decay of the localization error instead of merely exponential decay. Although superexponential convergence hinges on a conjecture of spectral geometry~\cite{HaPe21b}, it is provably at least as effective as LOD. More importantly, its sharper localization yields significant computational gains, reducing the diameter of the local patch problems from $\mathcal{O}(H \,\lvert \log H\rvert)$ down to $\mathcal{O}(H \sqrt{\lvert \log H\rvert})$. In turn, this improvement enhances the sparsity of the resulting system by a factor of $(|\log H|)^{d/2}$ (where $d$ is the dimension of the spatial domain) and accelerates both the offline assembly and the online solution phases.

So far, SLOD has been explored for linear and nonlinear scalar PDEs, such as the nonlinear Schr\"odinger equation~\cite{PeWZi24}, parametric problems~\cite{BHP22}, stochastic homogenization~\cite{HaMoPe24}, Helmholtz equations~\cite{FrHaPe21}, convection-dominated diffusion~\cite{BoFrPe24}, and even spatial network models~\cite{graphSLOD}. In this paper, we \emph{extend} the superlocalization paradigm to \emph{vector-valued} PDEs, focusing on heterogeneous linear elasticity. We show that the method retains its superior localization properties in the elasticity context, providing substantial enhancements in practical performance compared to standard LOD~\cite{HeP16}. 

We also present a basic numerical analysis that demonstrates that SLOD is never worse than LOD in theory, while offering clear advantages in practice. Furthermore, we develop a new implementation of SLOD within the open-source \texttt{deal.II} finite element library. Our experience shows that integrating multiscale methods into modern high-performance finite element codes is feasible and beneficial for addressing the large-scale, heterogeneous elasticity problems commonly encountered as a building block for advanced models in engineering and the sciences.

The remainder of this paper is organized as follows. In Section~\ref{s:model} we introduce the linear elasticity equations and discuss their solution challenges in the presence of multiscale microstructures. Section~\ref{s:ideal} describes the construction of an \emph{ideal} numerical homogenization framework, serving as a guiding principle for subsequent sections. In Section~\ref{s:slod}, we detail the construction of superlocalized functions employed in the definition of the SLOD method for vector-valued problems. Section~\ref{s:anal} introduces the practical SLOD method and provides a stability and convergence analysis of it, including rigorous error estimates and condition number bounds. Finally, in Section~\ref{s:numexp}, we outline the implementation details in \texttt{deal.II} and present numerical experiments that confirm the efficiency and accuracy of SLOD on benchmark problems of heterogeneous linear elasticity.

\section{Mathematical model problem}\label{s:model}
In this section, we present the linear elasticity equations and describe the challenges that arise when solving these equations in the presence of multiscale microstructural features. Such features often necessitate exceptionally fine meshes in standard finite element schemes, motivating the development of numerical homogenization methods. An important aspect of these methods is the \emph{support locality} of the basis functions defining the approximate solution spaces, which helps to capture fine-scale effects at coarse scales without excessive computational overhead. Accordingly, we will later examine rapidly decaying basis functions, laying the foundation for the construction of multiscale approximation spaces with localized basis functions.
\subsection{Formulation of the Problem}
Let $\Omega \subset \mathbb{R}^d$ ($d=2$ or $3$) be a polytopal domain that represents the geometry of an elastic body. The linear elasticity problem with homogeneous Dirichlet boundary conditions seeks the displacement field $\mathbf{\ub}:\Omega \rightarrow \mathbb{R}^d$ and the (second-order) stress tensor $\sigma(\ub(\cdot)): \Omega \rightarrow \mathbb{R}^{d \times d}$ such that
\begin{equation}\label{Model}
\begin{array}{rcll}
- \nabla \cdot \sigma(\ub) &=& \boldsymbol{f} & \text{in }\Omega, \\
\sigma(\ub)&=& \mathbf{A}:\varepsilon(\mathbf{\ub}) & \text{in }\Omega, \\
\mathbf{\ub}&=&\mathbf{0} & \text{on }\partial \Omega, \\
\end{array}
\end{equation}
where $\varepsilon(\ub(\cdot)): \Omega \rightarrow \mathbb{R}^{d\times d}$ with
\begin{equation}
\varepsilon(\ub):=\frac{1}{2}(\nabla \ub + \nabla \ub^T)
\end{equation}
is the (second-order) linearized strain tensor accounting for small deformations of the elastic material, $\mathbf{A}: \Omega \rightarrow \mathbb{R}^{d \times d \times d \times d}$ is the (fourth-order) elastic tensor establishing the linear relationship between the strain and stress tensors and describing the microstructural properties of the material, and $\boldsymbol{f}:\Omega \rightarrow \mathbb{R}^d$ represents the body force per unit volume acting on the object. Here $(\nabla \vb)_{ij}= \frac{\partial\vb_i}{\partial x_j}$.

The double-dot product ``$:$'' between two tensors contracts over the last two indices of the first tensor and the first two indices of the second tensor. Consequently, $\sigma_{ij}(\ub)= \sum_{k,l=1}^d\mathbf{A}_{ijkl}\varepsilon_{kl}(\ub)$. For any two second-order-tensor-valued functions $\xi$,$\zeta : \Omega \rightarrow \mathbb{R}^{d\times d}$, we define
\begin{equation*}
(\xi,\zeta)_{L^2(\Omega)}:=\int_{\Omega} \xi:\zeta \dx = \int_{\Omega} \sum_{i,j=1}^d\xi_{ij}\zeta_{ij} \dx.
\end{equation*}
Similarly, for any two vector-valued functions $\wb,\vb:\Omega \rightarrow \mathbb{R}^{d}$, we set
\begin{equation*}
(\wb,\vb)_{L^2(\Omega)}:=\int_{\Omega} \wb \cdot \vb \dx.
\end{equation*}
The $L^2$-norm of a vector- or tensor-valued function is given by $$\|\cdot\|_{L^2(\Omega)}:=\sqrt{(\cdot,\cdot)_{L^2(\Omega)}}.$$

With $\mathbf{V} := H^1_0(\Omega;\mathbb{R}^d)$ used as both the trial and test spaces, and assuming that $\mathbf{A}$ is symmetric in the sense that
\[
  \mathbf{A}_{ijkl} = \mathbf{A}_{jikl} = \mathbf{A}_{ijlk} = \mathbf{A}_{klij}
  \quad
  \text{a.e.\ in } \Omega,
\]
the weak formulation of \eqref{Model} reads: find $\ub \in \mathbf{V}$ such that
\begin{equation}\label{weak_form}
  a(\ub,\vb) = (\boldsymbol{f},\,\vb)_{L^2(\Omega)}
  \quad\text{for all } \vb \in \mathbf{V},
\end{equation}
where
\begin{equation*}
  a(\ub,\vb)
  =
  \bigl(\mathbf{A} : \varepsilon(\ub),\varepsilon(\vb)\bigr)_{L^2(\Omega)}.
\end{equation*}
Throughout this paper, we assume that $\boldsymbol{f}\in L^2(\Omega;\mathbb{R}^d)$ and 
\(\mathbf{A}\in L^{\infty}(\Omega;\mathbb{R}^{d\times d\times d\times d})\) satisfy the 
uniform ellipticity and boundedness condition: there exist constants $\alpha, \beta > 0$ such that
\begin{equation}\label{A-alpha-beta-bounds}
  2\,\alpha \,\xi : \xi 
  \le
  (\mathbf{A} : \xi) : \xi
  \le
  \beta \,\xi : \xi
\end{equation}
for all symmetric tensors $\xi\in \mathbb{S}$, almost everywhere in $\Omega$.
Here, $\mathbb{S}$ denotes the set of all symmetric $d\times d$ tensors.

Since Dirichlet boundary conditions are prescribed on the entire boundary of $\Omega$, it follows from Korn's inequality~\cite{HeP16} that
\begin{equation}\label{Gradient-Strain-Dirichlet}
  \|\nabla \vb\|_{L^2(\Omega)}
  \le
  \sqrt{2}\,\|\varepsilon(\vb)\|_{L^2(\Omega)}
  \quad
  \text{for all } \vb \in \mathbf{V}.
\end{equation}
Then, combining \eqref{A-alpha-beta-bounds} and \eqref{Gradient-Strain-Dirichlet}, and noting 
that $\|\varepsilon(\vb)\|_{L^2(\Omega)} \le \|\nabla \vb\|_{L^2(\Omega)}$, we arrive at
\begin{equation}\label{Equivalence-Energy-H1semi}
  \alpha \,\|\nabla \vb \|^2_{L^2(\Omega)}
  \le
  a(\vb,\vb)
  \le
  \beta \,\|\nabla \vb\|^2_{L^2(\Omega)}
  \quad
  \text{for all } \vb \in \mathbf{V}.
\end{equation}

\begin{remark}
  In the case of an isotropic material, the elasticity tensor 
  $\mathbf{A}$ takes the form
  \[
    \mathbf{A}_{ijkl}
    =
    \mu\,(\delta_{ik}\,\delta_{jl} +\delta_{il}\,\delta_{jk})
    +
    \lambda\,\delta_{ij}\,\delta_{kl},
  \]
  where $\delta_{ij}$ denotes the Kronecker delta, and $\mu,\lambda$ are the Lam\'e parameters. 
  The corresponding stress tensor is given by
  \[
    \sigma (\ub)
    =
    2\,\mu\,\varepsilon(\ub)
    +
    \lambda\,(\nabla\cdot \ub)\,\mathbf{I},
  \]
  with $\mathbf{I}\in \mathbb{R}^{d\times d}$ the second order identity tensor.
\end{remark}

\begin{remark}
Using an approach similar to that of \cite{Brenner2022} it can be shown that we can consider $\alpha=1$ without loss of generality. 
\end{remark}

\subsection{Standard finite element discretization}
Let $\mathcal{T}_h$ be a Cartesian mesh of $\Omega$ with mesh size $h$, and let $V_h$ be the $\mathcal{Q}_1$ finite element space associated with $\mathcal{T}_h$ \cite{BrS08}. Furthermore, define $\mathbf{V}_h:=(V_h)^d$. Then the standard finite element approximation of the solution $\ub\in \mathbf{V}$ of \eqref{weak_form} is given by $\ub_h\in \mathbf{V}_h$ such that
\begin{equation}\label{Galerkin-approx}
a(\ub_h,\vb) = (\boldsymbol{f},\vb)_{L^2(\Omega)}\quad \text{for all }\vb \in \mathbf{V}_h.
\end{equation}

Under suitable regularity assumptions on $\ub$, we have the classic error estimate
\begin{equation}\label{H1-H2-estimate}
\|\ub-\ub_h\|_{H^1(\Omega)}\leq C_{\mathbf{A}}h\|\ub \|_{H^2(\Omega)},
\end{equation}
where $C_{\mathbf{A}}$ is a constant depending on the size of the coefficients in $\mathbf{A}$ 
(see \cite[Theorem~3.1]{HeP16}). Moreover, the term $\|\ub\|_{H^2(\Omega)}$ can and usually will grow with the oscillations of $\mathbf{A}$. 
In fact, under standard elliptic regularity arguments, one can show 
$$\|\ub\|_{H^2(\Omega)}\leq C(\ub,\Omega) \|\mathbf{A}\|_{W^{1,\infty}}.$$
The inequality is sharp, so if $\mathbf{A}$ varies rapidly, then $\|\ub\|_{H^2(\Omega)}$ can become arbitrarily large.

\section{Ideal numerical homogenization}\label{s:ideal}
For materials with strongly heterogeneous (multiscale) microstructures, the norm $\|\mathbf{A}\|_{W^{1,\infty}}$ can be very large. As indicated by the estimate 
\eqref{H1-H2-estimate}, a standard finite element discretization must then employ 
extremely fine meshes to resolve the fine-scale features and mitigate the impact of 
$\|\mathbf{A}\|_{W^{1,\infty}}$. This makes standard FEM prohibitively expensive in practice.

In the following, we address this challenge by developing \emph{numerical homogenization} 
methods. These methods aim to produce approximations with optimal decay rates for the 
discretization error, \emph{without} the pre-asymptotic effects inherent in standard FEM. 
Moreover, they require no regularity assumptions on the solution beyond $H^1$-regularity.

\subsection{Approximation space}
Let the \emph{solution operator} for problem \eqref{weak_form} be given by $$\mathcal{A}^{-1}: L^2(\Omega;\mathbb{R}^d)\rightarrow H^1_0(\Omega;\mathbb{R}^d),$$
so that $\ub = \mathcal{A}^{-1}\fb$ is the exact solution. 
Similarly, define the \emph{standard finite element solution operator} 
$\mathcal{A}^{-1}_h: L^2(\Omega;\mathbb{R}^d)\rightarrow \mathbf{V}_h$, 
such that $\ub_h = \mathcal{A}^{-1}_h\fb$ is the standard FEM approximation. 
Throughout the paper, we assume that the fine mesh $\mathcal{T}_h$ is sufficiently 
refined to resolve the fine-scale variations of $\mathbf{A}$, making 
$\|\mathcal{A}^{-1}\fb - \mathcal{A}^{-1}_h\fb\|_a$ small. Here, 
$\|\cdot\|_a := \sqrt{a(\cdot,\cdot)}$ is the \emph{energy norm}.

Now let $\mathcal{T}_H$ be a Cartesian mesh of $\Omega$ with mesh size $H > h$, 
obtained by coarsening $\mathcal{T}_h$. We define
$$\mathcal{V}_H:=\mathrm{span}\{\mathcal{A}^{-1}_h (\eb_k \chi_T ) : T \in \mathcal{T}_H,\; k\in\{1,\ldots,d\}\},$$ where $\eb_k \in \mathbb{R}^d$ is the canonical vector whose $k$-th entry equals $1$ and the remaining equal zero, and $\chi_T$ is the characteristic function of $T$. Let $$\mathcal{A}^{-1}_H: L^2(\Omega;\mathbb{R}^d)\rightarrow \mathcal{V}_H$$ be the operator defining the approximate solution of \eqref{weak_form} in $\mathcal{V}_H$ such that
\begin{equation}\label{Ideal-coarse-solution-weak-form}
a(\mathcal{A}^{-1}_H\fb,\vb)=(\fb,\vb)_{L^2(\Omega)}\quad \text{for all } \vb \in \mathcal{V}_H,\; \fb \in L^2(\Omega;\mathbb{R}^d).
\end{equation}
By construction, if $\fb$ lies in $(\mathbb{Q}^0(\mathcal{T}_H))^d$ (i.e., its components are piecewise-constant functions with respect to the coarse mesh), then
$$\mathcal{A}^{-1}_h \fb = \mathcal{A}^{-1}_H \fb.$$ In other words, letting $$\Pi_H: L^2(\Omega;\mathbb{R}^d)\rightarrow (\mathbb{Q}^0(\mathcal{T}_H))^d$$ 
be the $L^2$-projection onto piecewise constants, we have $$\mathcal{A}^{-1}_h \circ \Pi_H = \mathcal{A}^{-1}_H \circ \Pi_H.$$ Hence, if $\|\fb - \Pi_H\fb\|_{L^2(\Omega)}$ is small, then 
$\mathcal{A}^{-1}_H\fb$ is a good approximation of $\mathcal{A}^{-1}_h\fb$ 
and, by extension, of $\mathcal{A}^{-1}\fb$ \cite{AHP21,MalP20}.

Observe that the stiffness matrix associated with \eqref{Ideal-coarse-solution-weak-form} 
is much smaller than the one in \eqref{Galerkin-approx}. In principle, this suggests 
solving \eqref{Ideal-coarse-solution-weak-form} rather than \eqref{Galerkin-approx}, 
since the linear system is substantially reduced while maintaining comparable accuracy. 
However, the functions $\mathcal{A}^{-1}_h(\eb_k \,\chi_T)$ generally have 
\emph{global support} and \emph{slow decay} (they are not negligible outside a small 
patch of $\Omega$). This makes their computation expensive and, in turn, 
renders the direct use of \eqref{Ideal-coarse-solution-weak-form} impractical.

Fortunately, one can construct an alternative basis for $\mathcal{V}_H$ whose 
functions \emph{decay rapidly} and can be approximated by \emph{locally supported} 
basis functions. Such \emph{localized} basis functions span an approximate solution 
space with nearly the same approximation properties as $\mathcal{V}_H$. 
We next focus on deriving these rapidly decaying basis functions and their 
efficient localized representations.
\subsection{Rapidly decaying basis functions}
We aim to construct basis functions of $\mathcal{V}_H$ in the form
\begin{equation}\label{varphi-levl2}
\varphi_{\tau,s}=\mathcal{A}^{-1}_{h}\gb_{\tau,s} \quad \mathrm{with}\quad \gb_{\tau,s} \coloneqq \sum_{k=1}^d\sum_{T\in \mathcal{T}_{H}}\eb_k c^{(\tau,s)}_{T,k}\chi_{T}.
\end{equation}
Here, $\tau \in \mathcal{T}_H$ and $s \in \{1,\ldots,d\}$. 
The vector of coefficients 
$\bigl(c_{T,k}^{(\tau,s)}\bigr)_{(T,k)\in \mathcal{T}_H \times \{1,\ldots,d\}}$ 
is chosen so that $\varphi_{\tau,s}$ decays rapidly away from the coarse element $\tau$. 
We refer to $\gb_{\tau,s}$ as the \emph{$\mathbb{Q}^0$-companion} of $\varphi_{\tau,s}$, 
while $\varphi_{\tau,s}$ is called the \emph{regularized companion} of $\gb_{\tau,s}$.

Before deriving these rapidly decaying basis functions, we introduce several definitions
adapted from \cite{BHP22,GarayMohrPeterseimetal}, which will be used in their construction. 
These definitions generalize the notion of local patches and support localization in the setting of vector-valued elasticity problems with heterogeneous coefficients.
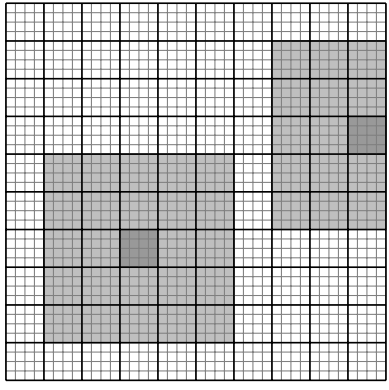
\begin{figure}[!ht]
\centering
  \begin{tikzpicture}
    \draw[fill=gray!50]  (0.5,0.5) -- (0.5,3) -- (3,3) -- (3,0.5) -- cycle;
    \draw[fill=gray!50]  (3.5,2) -- (3.5,4.5) -- (5,4.5) -- (5,2) -- cycle;
  \draw[fill=gray!80]  (1.5,1.5) -- (2,1.5) -- (2,2) -- (1.5,2) -- cycle;
  \draw[fill=gray!80]  (4.5,3) -- (4.5,3.5) -- (5,3.5) -- (5,3) -- cycle;
  \draw[gray] (0, 0) grid[step=0.125cm] (5, 5);
  \draw [black, semithick](0, 0) grid[step=0.5cm] (5, 5);
  \end{tikzpicture}
  \caption{Illustration of two second-order patches $\omega^{(2)}_T$, one built around an element far-enough from the boundary and another around an element touching the boundary. The fine mesh is presented in gray.}
\end{figure}
Define the $m$-th order patch $\omega_T^{(m)}\subset\Omega$  of an element $T\in \mathcal{T}_{H}$ recursively as
\begin{equation*}
\omega_{T}^{(m)}=\bigcup \{K\in \mathcal{T}_{H}: K\cap \omega^{(m-1)}_{T}\neq \emptyset\} \quad \mathrm{with}\quad \omega^{(0)}_{T}=T\in \mathcal{T}_{H}.
\end{equation*}
In other words, \(\omega_T^{(m)}\) is the union of all coarse elements 
whose intersection with \(\omega_T^{(m-1)}\) is nonempty, starting from 
the initial patch \(\omega_T^{(0)} = T\). 

For the remainder of this subsection, we fix \(m > 1\) and \(\tau \in \mathcal{T}_H\), 
and let \(\omega := \omega_{\tau}^{(m)}\). 
This choice of \(\omega\) will serve as the region in which subsequent constructions (e.g., local correctors, localized basis functions) are supported or computed.

Denote the restriction of $\mathcal{T}_{H}$ to the patch $\omega$ by $\mathcal{T}_{H,\omega}:=\{T\in\mathcal{T}_{H}:T\subset \omega\}$. Similarly, the restriction of $\mathbf{V}_h$ to $\omega$ is given by 
$$\mathbf{V}_{h,\omega}:=\{\vb\in H^1_0(\omega;\mathbb{R}^d):\vb_{|_{T}}\in (\mathbb{Q}^{1}(T))^d\; \text{for all } T\in \mathcal{T}_{h,\omega}\}.$$ Let $$\widetilde{\mathbf{V}}_{h,\omega}:=\{\vb\in H^1(\omega;\mathbb{R}^d):\vb_{|_{T}}\in (\mathbb{Q}^{1}(T))^d\; \text{for all } T\in \mathcal{T}_{h,\omega}\}.$$ Thus, the trace of a function in $\widetilde{\mathbf{V}}_{h,\omega}$ may be non-zero. With $\Sigma_{\omega}=\partial \omega \setminus \partial \Omega$, define the trace operator
\begin{equation*}
\text{tr}_{\Sigma_{\omega}}: \widetilde{\mathbf{V}}_{h,\omega}\rightarrow X_{\omega}:=\text{image tr}_{\Sigma_{\omega}}\subset H^{1/2}(\Sigma_{\omega};\mathbb{R}^d).
\end{equation*}
Note that the space $X_{\omega}$ can be equipped with the norm
\begin{equation*}
\|\wb\|_{X_{\omega}}:=\text{inf}\left\{ \|\vb\|_{H^{1}(\omega)}: \vb\in \widetilde{\mathbf{V}}_{h,\omega}, \; \text{tr}_{\Sigma_{\omega}}\vb=\wb \right\},
\end{equation*}
where $\|\vb\|_{H^1(\omega)}^2=\|\vb\|_{L_2(\omega)}^2+\|\nabla \vb\|_{L_2(\omega)}^2$.

By definition of the $\|\cdot\|_{X_{\omega}}$ norm, the continuity of the trace operator holds regardless of the patch geometry, i.e., 
\begin{equation}\label{trace-norm-continuity}
\|\text{tr}_{\Sigma_{\omega}}\vb\|_{X_{\omega}}\leq\|\vb\|_{H^{1}(\omega)} \;\;\;\;\;\;\text{for all } \vb\in \widetilde{\mathbf{V}}_{h,\omega}\subset H^1(\omega;\mathbb{R}^d).
\end{equation}

Let $a_{\omega}(\cdot,\cdot):H^{1}(\omega;\mathbb{R}^d) \times H^{1}(\omega;\mathbb{R}^d) \rightarrow \mathbb{R}$ such that $a_\omega(\ub,\vb)\coloneqq (\mathbf{A}:\varepsilon(\ub),\varepsilon(\vb))_{L^2(\omega)}$,
and let the linear operator $\mathcal{A}^{-1}_{h,\omega}:L^{2}(\omega;\mathbb{R}^d)\rightarrow \mathbf{V}_{h,\omega}$ be given by
\begin{equation*}
a_\omega(\mathcal{A}^{-1}_{h,\omega} \gb,\vb )=(\gb,\vb)_{L^2(\omega)},\;\;\; \text{for all } \vb\in \mathbf{V}_{h,\omega}, \text{ and } \gb\in L^{2}(\omega;\Rd).
\end{equation*}
Denote by $\mathcal{E}_\omega: L^2(\omega;\Rd) \rightarrow L^2(\Omega;\Rd)$ 
an extension by zero operator. With 
$$\gb=\sum_{k=1}^d\sum_{T\in \mathcal{T}_{H,\omega}}\eb_k c_{T}\chi_{T}\in (\mathbb{Q}^{0}(\mathcal{T}_{H,\omega}))^d,$$
consider 
\begin{equation*}
 \bar{\psi}_{\gb}:=\mathcal{A}^{-1}_{h,\omega}\gb \in \mathcal{V}_{H,\omega}\quad \mathrm{and}\quad \psi_{\gb}:=\mathcal{A}^{-1}_{h}\mathcal{E}_\omega(\gb) \in \mathcal{V}_{H},
\end{equation*}
where $\mathcal{V}_{H,\omega}=\{\mathcal{A}^{-1}_{h,\omega}\eb_k \chi_T : T\in \mathcal{T}_{H,\omega},\;k\in\{1,\ldots,d\}\}$.
We conclude the set of definitions by defining the {\it traction} conormal derivative of $\bar{\psi}_{\gb}$ as the functional $\gamma_{\bar{\psi}_{\gb}}:X_\omega \rightarrow \mathbb{R}$ such that
\begin{equation}\label{Def-Conormal-Der}
\langle \gamma_{\bar{\psi}_g},\text{tr}_{\Sigma_{\omega}}(\vb) \rangle=a_{\omega}(\bar{\psi}_{\gb},\vb)-(\gb,\vb)_{L^2(\omega)}\;\;\;\;\text{for all } \vb\in H^{1}(\omega;\Rd).
\end{equation}
From the above definition, we have $$\langle \gamma_{\bar{\psi}_g},\text{tr}_{\Sigma_{\omega}}(\vb) \rangle = (\sigma(\bar{\psi}_g)\nu,\text{tr}_{\Sigma_{\omega}}(\vb))_{L_2(\Sigma_{\omega})},$$ where $\nu$ is the outer normal unit vector on $\Sigma_{\omega}$.

Note that $\mathcal{E}_{\omega}(\bar{\psi}_{\gb})$ does not necessarily belong to $\mathcal{V}_{H}$. However, if $\gb$ is such that $\mathcal{E}_{\omega}(\bar{\psi}_{\gb})=\psi_{\gb}$, then $\psi_{\gb}$ would be a locally supported function of $\mathcal{V}_{H}$ whose $\mathbb{Q}^{0}$-companion $\mathcal{E}_{\omega}(g)$ is also locally supported. Furthermore, if such equality holds with a small $\omega$ (i.e., $m$ is small) we would say that $\psi_{\gb}$ is rapidly decaying. 
  
In the following lemma, we determine a bound for the energy norm of the {\it localization error} $\mathcal{E}_{\omega}(\bar{\psi}_{\gb})-\psi_{\gb}$, which depends on the $X_{\omega}'$-norm of $\gamma_{\bar{\psi}_{\gb}}$ given by
\begin{equation}\label{ConormalDer-Norm-definition}
\|\gamma_{\bar{\psi}_{\gb}}\|_{X_{\omega}'}=\sup_{\wb\in X_\omega\setminus \{0\}}\frac{\langle \gamma_{\bar{\psi}_{\gb}},\wb \rangle}{\|\wb\|_{X_\omega}},
\end{equation}
providing a way to measure the dependence of the error on the coefficients defining $\gb$. The proof is identical to the scalar PDE case given in \cite{GarayMohrPeterseimetal}.
\begin{lemma}\label{Lemma-Localization-Error-bound-conormalDer}
Let $\bar{\psi}_{g}$, $\psi_g$, and $\gamma_{\bar{\psi}_{g}}$ be defined as above. Then, the energy norm of the localization error has the bound
\begin{equation}\label{Loc-Error-bound-conormalDer}
\|\mathcal{E}_\omega(\bar{\psi}_{\gb})-\psi_{\gb}\|_{a}\leq \frac{1}{\sqrt{\alpha}}\biggr(1+\frac{\mathrm{diam}(\Omega)}{\pi}\biggl)\|\gamma_{\bar{\psi}_{\gb}}\|_{X_{\omega}'},
\end{equation}
where $\mathrm{diam}(\Omega)$ denotes the diameter of $\Omega$ and the constant $\alpha$ is given in \eqref{A-alpha-beta-bounds}.
\end{lemma}
\begin{proof}
We have $\text{for all } \vb \in H^{1}(\Omega;\Rd)$ that
\begin{eqnarray}\label{Loc-Error-bound-conormalDer-prelim-result}
\nonumber a(\mathcal{E}_{\omega}(\bar{\psi}_{\gb})-\psi_{\gb},\vb)&=&a_{\omega}(\bar{\psi}_{\gb},\vb_{|_{\omega}})-(\gb,\vb_{|_{\omega}})_{L^2(\omega)} \\ \nonumber
&=&\langle \gamma_{\bar{\psi}_{\gb}},\text{tr}_{\Sigma_{\omega}}(\vb_{|_{\omega}}) \rangle \\ \nonumber
&\leq & \|\gamma_{\bar{\psi}_{\gb}}\|_{X_{\omega}'} \|\text{tr}_{\Sigma_{\omega}}(\vb_{|_{\omega}})\|_{X_{\omega}} \\ \nonumber
&\leq &\|\gamma_{\bar{\psi}_{\gb}}\|_{X_{\omega}'}\|\vb_{|_{\omega}}\|_{H^{1}(\omega)}\\ 
&\leq & \frac{1+\frac{\mathrm{diam}(\omega)}{\pi}}{\sqrt{\alpha}}\|\gamma_{\bar{\psi}_{\gb}}\|_{X_{\omega}'} \|\vb_{|_{\omega}}\|_{a_{\omega}},
\end{eqnarray}
where the first inequality comes from the definition of the $X'_{\omega}$-norm, the second inequality from (\eqref{trace-norm-continuity}), and the last inequality is obtained using the Poincar\'e-Friedrichs inequality and the energy norm bound given in \eqref{Equivalence-Energy-H1semi}. 

Then, with $\vb=\mathcal{E}_{\omega}(\bar{\psi}_{\gb})-\psi_{\gb}$ in \eqref{Loc-Error-bound-conormalDer-prelim-result}, and since $\|\vb_{|_{\omega}}\|_{a_{\omega}}\leq \|\vb\|_{a}$ and $\mathrm{diam}(\omega)\leq \mathrm{diam}(\Omega)$, the inequality \eqref{Loc-Error-bound-conormalDer} is obtained.
\end{proof}

It follows from Lemma~\ref{Lemma-Localization-Error-bound-conormalDer} that, if there exists a non-trivial $\gb\in (\mathbb{Q}^0(\mathcal{T}_{H,\omega}))^d$ such that $\|\gamma_{\bar{\psi}_{\gb}}\|_{X_{\omega}'}$ equals $0$, then there exists a locally supported function of $\mathcal{V}_{H}$ with locally supported $\mathbb{Q}^0$-companion $\mathcal{E}_\omega(\gb)$.

In general, it is not always possible to construct locally supported functions whose traction 
conormal derivatives (in the dual norm) vanish completely. Nevertheless, one \emph{can} obtain 
locally supported functions for which these quantities are merely \emph{small}. Numerical 
experiments reveal that certain basis functions with small traction conormal derivatives 
may lead to ill-conditioning, since they can lack sufficient linear independence. 
Hence, our objective is to identify localized functions that simultaneously feature small localization 
error \emph{and} maintain basis stability.

To achieve a stable basis, we seek locally supported functions 
\(\varphi_{T,k} \in H^1_0(\Omega;\Rd)\) that satisfy the 
\emph{basis stability condition}
\begin{equation}\label{basis-stability-condition}
\Bigg\|\frac{\Pi_{{H}}\varphi_{T,k}}{z_{T,k}}-\eb_{k}\chi_{T}\Bigg\|_{L^{\infty}(\Omega)}\leq\delta_s,
\end{equation} 
where
\begin{equation*}
z_{T,k}=\frac{\big(\Pi_{{H}}\varphi_{T,k},\eb_{k}\chi_{T}\big)_{L^2(\Omega)}}{\left(\eb_{k}\chi_{T},\eb_{k}\chi_{T}\right)_{L^2(\Omega)}},
  \quad
  T \in \mathcal{T}_H,\quad
  k \in \{1,\ldots,d\},
\end{equation*}
and \(\delta_s \ge 0\) is a small parameter (in practice, taking \(\delta_s \le 0.5\) 
often suffices to ensure stability).
Condition 
\eqref{basis-stability-condition} forces \(\varphi_{T,k}\) to remain relatively 
concentrated around \(T\) and predominantly in its \(k\)-th component. 
This property promotes enough linear independence among the basis functions 
to stabilize the resulting system; see \cite{GarayMohrPeterseimetal} 
for further details on why \eqref{basis-stability-condition} implies basis stability.
In the next section, we develop superlocalized basis functions that satisfy 
\eqref{basis-stability-condition}, yielding a stable basis whose localization errors 
decay \emph{superexponentially} with $m$.

\section{Construction of a stable superlocalized basis}\label{s:slod}

In this section, we detail the construction of such superlocalized basis functions. 
We first introduce a variant of the classical LOD method 
that produces a stable basis with \emph{exponentially} decaying functions. 
Building on the ideas from the previous discussion, this LOD variant ensures the 
LOD basis functions act as (locally) regularized companions of functions in 
$(\mathbb{Q}^0(\Omega))^d$. We then present a \emph{superlocalization} strategy, 
extending the LOD theory to achieve basis functions whose localization errors 
decay superexponentially in $m$.

\subsection{Construction of the LOD basis}\label{section:SLOD}
Using the energy-minimization saddle-point formulation of the LOD method \cite{Mai20ppt,OwhS19}, we have for every $T\in \mathcal{T}_{H,\omega}$ and $k\in \{1,\ldots,d\}$ that the local LOD function $\bar{\psi}^{\text{LOD}}_{T,k}$ solves
\begin{equation}\label{LOD-Saddle-point-formulation}
\begin{pmatrix}
\mathcal{A}_{h,\omega} & \mathcal{P}^T\\
\mathcal{P} & \mathbf{0}\\
\end{pmatrix}
\begin{pmatrix}
\bar{\psi}^{\text{LOD}}_{T,k}\\
\lambda
\end{pmatrix}= \begin{pmatrix}
\mathbf{0}\\
\mathbf{e}_k \chi_{T} 
\end{pmatrix},
\end{equation}
where $\mathcal{A}_{h,\omega}:\Vh_{h,\omega}\rightarrow \left[\Vh_{h,\omega}\right]'$, $\wb \mapsto a_{\omega}(\wb,\cdot)$, $\mathcal{P}:\Vh_{h,\omega}\rightarrow \big(\mathbb{Q}^{0}(\mathcal{T}_{H,\omega})\big)^d$, $\wb \mapsto \Pi_{H,\omega}\wb$, and $\mathcal{P}^{T}:\big(\mathbb{Q}^{0}(\mathcal{T}_{H,\omega})\big)^d\rightarrow \left[\Vh_{h,\omega}\right]'$ such that $\langle \mathcal{P}^{T} \boldsymbol{p}, \vb \rangle = (\boldsymbol{p},\vb)_{L^2(\omega)}$ for all $\boldsymbol{p}\in (\mathbb{Q}^{0}(\mathcal{T}_{H,\omega}))^d$ and $\vb\in \Vh_{h,\omega}$. Let $\mathcal{D}:\big(\mathbb{Q}^{0}(\mathcal{T}_{H,\omega})\big)^d\rightarrow \big(\mathbb{Q}^{0}(\mathcal{T}_{H,\omega})\big)^d$ such that $\mathcal{D}=\mathcal{P}\circ \mathcal{A}_{h,\omega}^{-1}\circ \mathcal{P}^T$. It can be shown that $\mathcal{D}$ is invertible; see \cite{HaPe21b}. Thus, eliminating $\lambda$ in \eqref{LOD-Saddle-point-formulation} and solving for $\bar{\psi}^{\text{LOD}}_{T,k}$ yields
\begin{equation}\label{Local LOD function}
\bar{\psi}^{\text{LOD}}_{T,k}=\mathcal{A}^{-1}_{h,\omega}\gb_{T,k} \quad \mathrm{with} \quad \gb_{T,k} =\mathcal{D}^{-1}\eb_k\chi_{T}.
\end{equation}

With the same arguments as in \cite[Theorem 4.1]{Mai20ppt} and \cite[Theorem 3.15]{AHP21} but with vector-valued cut-off functions instead of scalar ones (as in \cite[Section 4]{HeP16}), and using the inverse estimate $\|\nabla \vb\|_{L^2(\omega)}\leq C_{\sharp} H^{-1}\|\vb\|_{L^2(\omega)}$ for all $\vb \in \Vh_{h,\omega}$ (with $C_{\sharp}$ independent of $H$), we arrive at the following theorem expressing the exponential decay of local {\em vector-valued} LOD basis functions.

\begin{theorem}\label{LOD basis decay}
Consider $\omega=\omega_T^{(m)}$ for $T\in \mathcal{T}_H$ and $m>1$, and let $\bar{\psi}^{\text{LOD}}_{T,k}\in \Vh_{h,\omega}$ and $\gb_{T,k}\in (\mathbb{Q}^0(\omega))^d$ be defined as in \eqref{Local LOD function}. Then for $r\leq m$ it holds that
\begin{equation}
\|\nabla \bar{\psi}^{\text{LOD}}_{T,k}\|_{L^2(\omega\setminus \omega_{T}^{(r)})}\leq C_{\dagger}H^{-1}e^{-Cr}\|\gb_{T,k}\|_{L^2(\omega)},
\end{equation}
where the constants $C_{\dagger}$ and $C$ depend on $\alpha$ and $\beta$ but do not depend on $H$, $m$, and $r$.
\end{theorem}

Using \cref{LOD basis decay} together with the arguments given in \cite[Lemma A.2]{HaPe21b}, we obtain the following result for the traction conormal derivative of vector-valued LOD basis functions.

\begin{lemma}
Let $\bar{\psi}^{\text{LOD}}_{T,k}\in \Vh_{h,\omega}$ and $\gb_{T,k}\in (\mathbb{Q}^0(\omega))^d$ be defined as in \eqref{Local LOD function} and $\omega=\omega_{T}^{(m)}$. Then the dual norm of the traction conormal derivative of $\bar{\psi}^{\text{LOD}}_{T,k}$ has the bound
\begin{equation}
\|\gamma_{\bar{\psi}^{\text{LOD}}_{T,k}}\|_{X'_{\omega}}\leq C_{*}H^{-1}e^{-Cm},
\end{equation}
where the constants $C_{*}$ and $C$ depend on $\alpha$ and $\beta$ but do not depend on $H$ and $m$.
\end{lemma}

From the second line of \eqref{LOD-Saddle-point-formulation}, it follows that the extension by zero of the local LOD basis functions give rise to a stable basis of an approximate solution space of \eqref{weak_form}, satisfying the stability condition \eqref{basis-stability-condition} with $\delta_s=0$.

\subsection{Superlocalization of basis functions} 
Consider $\widetilde{T}\in \mathcal{T}_{H,\omega}$, $\widetilde{k} \in \{1,\ldots,d\}$, and $m>1$ fixed, and let $\omega:=\omega_{\widetilde{T}}^{(m)}$. Let $S:=\mathcal{T}_{H,\omega} \times \{1,\ldots,d\} \setminus \big\{(\widetilde{T},\widetilde{k})\big\}$. For $(T,k)\in S $, let $\bar{\psi}_{T,k}^{(\widetilde{T},\widetilde{k})}:=\mathcal{A}^{-1}_{h,\omega}\gb_{T,k}$, where $\gb_{T,k}=\mathcal{D}^{-1}\eb_k\chi_{T}\in (\mathbb{Q}^0(\omega))^d$. With $\mathbf{c}=(c_{T,k})_{(T,k)\in S}$, define $\Psi_{\mathbf{c}}^{(\widetilde{T},\widetilde{k})}\in \Vh_{h,\omega}$ as
\begin{equation*}
\Psi_{\mathbf{c}}^{(\widetilde{T},\widetilde{k})}:=\sum_{(T,k)\in S}c_{T,k} \bar{\psi}_{T,k}^{(\widetilde{T},\widetilde{k})}=\mathcal{A}^{-1}_{h,\omega}\sum_{(T,k) \in S}c_{T,k} \gb_{T,k}.
\end{equation*}
Consider $n_e=\#T_{H,\omega}$. We want to find $\bar{\mathbf{c}}\in \mathbb{R}^{n_e d-1}$ such that $\|\gamma_{\bar{\psi}^{\text{LOD}}_{\widetilde{T},\widetilde{k}}+\Psi_{\bar{\mathbf{c}}}^{(\widetilde{T},\widetilde{k})}}\|_{X'_{\omega}}$ is as small as possible (ideally zero), and then define the (normalized) SLOD basis function associated with  $(\widetilde{T},\widetilde{k})\in \mathcal{T}_H \times \{1,\ldots,d\}$ as
\begin{equation}\label{SLOD basis def}
\hat{\varphi}^{\text{SLOD}}_{\widetilde{T},\widetilde{k}}=\frac{\mathcal{E}_{\omega}\left(\bar{\psi}^{\text{LOD}}_{\widetilde{T},\widetilde{k}}+\Psi_{\bar{\mathbf{c}}}^{(\widetilde{T},\widetilde{k})}\right)}{\|\bar{\psi}^{\text{LOD}}_{\widetilde{T},\widetilde{k}}+\Psi_{\bar{\mathbf{c}}}^{(\widetilde{T},\widetilde{k})}\|_{a_{\omega}}}\cdot
\end{equation}
To make $\|\gamma_{\bar{\psi}^{\text{LOD}}_{\widetilde{T},\widetilde{k}}+\Psi_{\mathbf{c}}^{(\widetilde{T},\widetilde{k})}}\|_{X'_{\omega}}$ small, we could use a similar technique as in \cite{BHP22} to compute the $\|\cdot\|_{X'_{\omega}}$ norm, and then minimize it over the coefficients $\mathbf{c}\in \mathbb{R}^{n_e d-1}$ to obtain $\bar{\mathbf{c}}$. However, for computational-cost efficiency, and based on \eqref{ConormalDer-Norm-definition}, we rather seek $\bar{\mathbf{c}}\in \mathbb{R}^{n_e d-1}$ such that 
\begin{equation}\label{Weekened-boundary-condition}
\sum_{\substack{i\in I_{\Sigma_{\omega}},\\ k\in \{1,\ldots,d\}}}\left\langle \gamma_{\bar{\psi}^{\text{LOD}}_{\widetilde{T},\widetilde{k}}+\Psi_{\bar{\mathbf{c}}}^{(\widetilde{T},\widetilde{k})}},\mathrm{tr}_{\Sigma_{\omega}}\eb_k\phi_i \right\rangle ^2\leq\epsilon
\end{equation}
for a small $\epsilon\geq 0$, where $\Sigma_{\omega}:=\partial \omega\setminus \partial \Omega$, $I_{\Sigma_{\omega}}:=\{i\in \mathbb{N}:x_i\in \Sigma_{\omega}\}$, $x_i$ is the $i$-th nodal point associated with $\mathcal{T}_{h,\omega}$, and $\phi_i$ is the $i$-th $\mathcal{Q}_{1}$ standard basis function of $\widetilde{V}_h:=\{v\in H^1(\omega;\mathbb{R}):v_{|_{T}}\in \mathbb{Q}^{1}(T)\; \text{for all } T\in \mathcal{T}_{h,\omega}\}$ associated with $x_i$.


\begin{remark}
In what follows, whenever $T$ appears without parentheses in a superscript, it denotes the transpose sign.
\end{remark}

With $n_b=\#I_{\Sigma_{\omega}}$, define the matrix $\mathbf{B}\in \mathbb{R}^{(n_{b} d) \times (n_e d)}$ such that
\begin{equation}\label{Bmatrix-TraceCondition}
\mathbf{B}_{ij} = a_{\omega}(\mathcal{A}^{-1}_{h,\omega}\eb_k\chi_{T_q},\eb_s\phi_p)-(\eb_k\chi_{T_q},\eb_s\phi_p)_{L^2(\omega)},
\end{equation}
where $i=n_b(s-1)+p$ and $j=n_e(k-1)+q$. Let $g_{\tau,s}=\sum_{k=1}^d\big(\sum_{T\in \mathcal{T}_{{H,\omega}}}d_{T,k}^{(\tau,s)}\chi_T\big)\eb_k$, with $\tau \in \mathcal{T}_{H,\omega}$ and $s\in\{1,\ldots,d\}$, $\mathbf{d}^{(\tau,s)}=\big(d_{T,k}^{(\tau,s)}\big)_{(T,k)\in \mathcal{T}_{H,\omega}\times \{1,\ldots,d\}}$, and $\mathbf{D}\in \mathbb{R}^{n_e d \times (n_e d -1)}$ such that the $j$-th column of $\mathbf{D}$ is $\mathbf{d}^{(\tau,s)_j}$, with $(\tau,s)_j \in S$. Then, the $\bar{\mathbf{c}}$ providing the smallest $\epsilon$ in \eqref{Weekened-boundary-condition} is the least-squares-error solution of
\begin{equation}\label{linear-system-solved-by-least-squares}
\mathbf{B}\mathbf{D}\mathbf{c}=-\mathbf{B}\mathbf{d}^{\left(\widetilde{T},\widetilde{k}\right)}
\end{equation}
i.e.,
\begin{equation}\label{c-lsq-solution}
\bar{\mathbf{c}}=-\left((\mathbf{B}\mathbf{D})^{T}\mathbf{B}\mathbf{D}\right)^{-1}(\mathbf{B}\mathbf{D})^{T}\mathbf{B}\mathbf{d}^{\left(\widetilde{T},\widetilde{k}\right)}.
\end{equation}
Note that, to guarantee stability of the SLOD basis, we additionally want $\Psi_{\bar{\mathbf{c}}}$ to be such that $\hat{\varphi}^{\text{SLOD}}_{\widetilde{T},\widetilde{k}}$ satisfies \eqref{basis-stability-condition}.
In practice, it is observed that choosing $\bar{\mathbf{c}}$ as in \eqref{c-lsq-solution} does not always satisfy condition \eqref{basis-stability-condition}. Hence, for the sake of stability, we choose $\bar{\mathbf{c}}$ instead as follows. Expressing $(\mathbf{B}\mathbf{D})^{T}\mathbf{B}\mathbf{D}$ in terms of its singular value decomposition we have
\begin{equation}\label{svd-computation}
(\mathbf{B}\mathbf{D})^{T}\mathbf{B}\mathbf{D}=\sum_{i=1}^{r}\sigma_i u_i v_i^T.
\end{equation}
where $\sigma_i$ is the $i$-th singular value, with $\sigma_1\geq\ldots \geq\sigma_r$, $u_i$ is the $i$-th left singular vector, $v_i$ is the $i$-th right singular vector, and $r$ is the rank of $(\mathbf{B}\mathbf{D})^{T}\mathbf{B}\mathbf{D}$. Then, we can make a stable choice of $\bar{\mathbf{c}}$ by taking 
\begin{equation}\label{stable_c_computation}
\bar{\mathbf{c}}_{s}=-\Big(\sum_{i=1}^{r_s^{(\widetilde{T},\widetilde{k})}}\sigma_i^{-1} v_i u_i^T\Big)(\mathbf{B}\mathbf{D})^{T}\mathbf{B}\mathbf{d}^{\left(\widetilde{T},\widetilde{k}\right)},
\end{equation}
where $r_s^{(\widetilde{T},\widetilde{k})}\leq r$ is chosen so that condition \eqref{basis-stability-condition} holds. In fact, for $r_s^{(\widetilde{T},\widetilde{k})}\geq1$ we have
\begin{equation}
\|\bar{\mathbf{c}}_s\|_2\leq \Big\|\sum_{i=1}^{r_s^{(\widetilde{T},\widetilde{k})}}\sigma_i^{-1} v_i u_i^T\Big\|_2 \|\mathbf{B}\mathbf{D}\|_2\|\mathbf{B}\mathbf{d}^{\left(\widetilde{T},\widetilde{k}\right)}\|_2= \sigma^{-1}_{r_s^{(\widetilde{T},\widetilde{k})}}\sqrt{\sigma_1} \|\mathbf{B}\mathbf{d}^{\left(\widetilde{T},\widetilde{k}\right)}\|_2,
\end{equation}
where
\begin{equation}
\|\mathbf{B}\mathbf{d}^{\left(\widetilde{T},\widetilde{k}\right)}\|_2 =\left(
\sum_{\substack{i\in I_{\Sigma_{\omega}},\\ k\in \{1,\ldots,d\}}}\left\langle \gamma_{\bar{\psi}^{\text{LOD}}_{\widetilde{T},\widetilde{k}}},\mathrm{tr}_{\Sigma_{\omega}}\eb_k\phi_i \right\rangle ^2\right)^{\frac{1}{2}}. 
\end{equation}
Thus, the entries of $\bar{\mathbf{c}}_s$ are expected to decrease as $r_s^{(\widetilde{T},\widetilde{k})}$ decreases, making it possible to satisfy the stability condition \eqref{basis-stability-condition}. 

Let $\mathbb{U}\in \mathbb{R}^{n_b d \times n_b d}$ be the (unitary) matrix whose $j$-th column is the $j$-th left singular vector of $\mathbf{B}\mathbf{D}$. Define
\begin{equation}\label{conormal-approx}
\mathbf{t}_{\mathbf{c}}:=\|\mathbf{B}\mathbf{D}\mathbf{c}+\mathbf{B}\mathbf{d}^{\left(\widetilde{T},\widetilde{k}\right)}\|_2=\left( \sum_{\substack{i\in I_{\Sigma_{\omega}},\\ k\in \{1,\ldots,d\}}}\left\langle \gamma_{\bar{\psi}^{\text{LOD}}_{\widetilde{T},\widetilde{k}}+\Psi_{{\mathbf{c}}}^{(\widetilde{T},\widetilde{k})}},\mathrm{tr}_{\Sigma_{\omega}}\eb_k\phi_i \right\rangle ^2\right)^{\frac{1}{2}}.
\end{equation}
Furthermore, let $\mathbf{R}\in \mathbb{R}^{n_b d \times n_b d}$ be the diagonal matrix such that $\mathbf{R}_{ii}=1$ for $r_{s}^{(\widetilde{T},\widetilde{k})}<i\leq r$ and $\mathbf{R}_{ii}=0$ otherwise. Then, from \eqref{conormal-approx}, the triangle inequality, \eqref{c-lsq-solution}, \eqref{stable_c_computation}, and with $\mathbf{y}:=\mathbb{U}^T\mathbf{B}\mathbf{d}^{(\widetilde{T},\widetilde{k})}$ it follows that
\begin{equation*}
\mathbf{t}_{\bar{\mathbf{c}}_s}-\mathbf{t}_{\bar{\mathbf{c}}}\leq \|\mathbf{B}\mathbf{D}\bar{\mathbf{c}}_s-\mathbf{B}\mathbf{D}\bar{\mathbf{c}}\|_2=\|\mathbb{U}\mathbf{R}\mathbb{U}^T\mathbf{B}\mathbf{d}^{(\widetilde{T},\widetilde{k})}\|_2\leq \|\mathbf{R}\mathbb{U}^T\mathbf{B}\mathbf{d}^{(\widetilde{T},\widetilde{k})}\|_2 = \left(\sum_{i=r_{s}^{(\widetilde{T},\widetilde{k})}+1}^{r}\mathbf{y}_{i}^2\right)^{\frac{1}{2}}.
\end{equation*}
This estimate implies that the difference in the basis localization error caused by using $\bar{\mathbf{c}}_s$ instead of $\bar{\mathbf{c}}$ increases as $r_{s}^{(\widetilde{T},\widetilde{k})}$ decreases, and the difference cannot be greater than $\|\mathbf{y}\|_2\leq \|\mathbf{B}\mathbf{d}^{\left(\widetilde{T},\widetilde{k}\right)}\|_2$ (up to a multiplicative constant).

Note that taking $r_s^{(\widetilde{T},\widetilde{k})}=0$ for all $T\in \mathcal{T}_{H}$ yields the LOD basis, which as mentioned above is a stable one. Therefore, there always exists at least one value of $r_s^{(\widetilde{T},\widetilde{k})}$ for which a stable basis can be obtained using this stabilization procedure. However, if $r_s^{(\widetilde{T},\widetilde{k})}$ is too small, the superexponential decay of basis functions might be lost, as expected. Nevertheless, even in such cases, the basis functions still exhibit exponentially decaying properties. Hence, we want to choose $\delta_s$ just small enough to achieve basis stability and such that $r_s^{(\widetilde{T},\widetilde{k})}$ remains sufficiently large to preserve the superlocalization properties. The value of $r_s^{(\widetilde{T},\widetilde{k})}$ is obtained by an iterative process which involves discarding the smallest singular value $\sigma_i$ in \eqref{stable_c_computation} at each iteration until condition \eqref{basis-stability-condition} is satisfied.

\section{Stability and convergence analysis of the SLOD Method}\label{s:anal}
Having described the construction of superlocalized basis functions, we are ready to introduce the practical SLOD method. 
Let $$\hat{\mathcal{V}}_H=\mathrm{span}\{\hat{\varphi}^{\text{SLOD}}_{T,k}\;:\;T\in \mathcal{T}_H,\;k\in\{1,\ldots,d\}\}$$
be the superlocalized approximation space. Then the SLOD method seeks $\hat{\ub}_H \in \hat{\mathcal{V}}_H$ such that
\begin{equation}\label{SLOD weak form}
a(\hat{\ub}_H,\vb)=(\fb,\vb)_{L^2(\Omega)}\quad \text{for all }\vb \in \hat{\mathcal{V}}_H.
\end{equation}
In what follows, we derive an estimate of the energy error between the SLOD approximation $\hat{\ub}_H$ and $\ub$ (solution of \eqref{weak_form}), as well as a condition number bound for the stiffness matrix associated with the SLOD basis. Together, these results establish the accuracy and stability of the SLOD method.

\subsection{Energy error estimate}
Before analyzing the error of our SLOD approximation, we state a lemma that bounds the 
$2$-norm of the coefficient vectors arising from the expansion of a function in terms of a basis. This bound will be essential 
both for deriving the SLOD error estimates and for establishing the condition number of the 
corresponding stiffness matrix. The argument follows from applying Rayleigh quotient bounds 
to the Gram matrix of a basis.

\begin{lemma}\label{Riesz-basis}
Let $\mathcal{B}=\{b_{i}\}_{i\in\{1,\ldots,n\}}$ be a basis of an inner product space $\mathcal{V}$ with norm $\|\cdot\|_{\mathcal{V}}$ induced by the inner product $(\cdot,\cdot)_{\mathcal{V}}$. Then $\{b_{i}\}_{i\in\{1,\ldots,n\}}$ is a Riesz basis, i.e., there exists constants $0 \leq C_1 \leq C_2$ such that for any finite sequence of real numbers $(c_i)_{i\in\{1,\ldots,n\}}$ we have
\begin{equation*}
C_1 \sum_{i=1}^{n}|c_i|^2\leq \left\|\sum_{i=1}^{n}c_i b_{i}\right\|_{\mathcal{V}}^2\leq C_2 \sum_{i=1}^{n}|c_i|^2.
\end{equation*}
Moreover, the bounds are tight by taking $C_1$ and $C_2$ as the smallest and largest eigenvalues of the basis Gram matrix $\widetilde{\mathbf{B}}\in \mathbb{R}^{n \times n}$ such that $\widetilde{\mathbf{B}}_{ij}=(b_{i},b_{j})_{\mathcal{V}}$.
\end{lemma}

Define 
\begin{equation}\label{sigma-def}
\widetilde{\sigma}:= \max_{(T,k)\in \mathcal{T}_H \times \{1,\ldots,d\}} \|\gamma_{\hat{\varphi}^{\mathrm{SLOD}}_{T,k_{|_{\omega_{T}}}}}\|_{X'_{\omega_{T}}}.
\end{equation}
Also, let $N_H:=\# \mathcal{T}_H$ and define
\begin{equation}\label{Basis single index}
\hat{\varphi}_i:=\hat{\varphi}^{\text{SLOD}}_{T_q,k}\quad\quad\text{ and } \quad \quad\; \varphi_i:=\mathcal{A}_h \gb^{\text{SLOD}}_{T_q,k},
\end{equation}
where $i=(k-1)N_H+q$, $T_q\in \mathcal{T}_H$, $k\in\{1,\ldots,d\}$, $\hat{\varphi}^{\text{SLOD}}_{T_q,k}$ is as defined in \eqref{SLOD basis def}, and $\gb^{\text{SLOD}}_{T_q,k}$ is the extension by zero of the $\mathbb{Q}^0$-companion of the restriction of $\hat{\varphi}^{\text{SLOD}}_{T_q,k}$ to its supporting patch. Further, consider $\hat{\mathcal{B}}:=\{\hat{\varphi}_i\;:\;i=1,\ldots,N_H d\}$.

Now we are ready to present the theorem providing the error estimate for the approximate solution of \eqref{weak_form} obtained with the SLOD basis.

\begin{theorem}
Let $\ub \in H^1_0(\Omega;\Rd)$ be the solution of \eqref{weak_form}, $\ub_h\in \Vh_h(\Omega)$ the solution of \eqref{Galerkin-approx}, and $\hat{\ub}_H$ the solution of \eqref{SLOD weak form}. Then, the following approximation estimate holds. 
\begin{eqnarray}\label{S-hat-Error-Estimate}
\nonumber \|\ub-\hat{\ub}_H\|_a\leq \|\ub-\ub_h\|_a &+& \frac{H}{\pi \sqrt{\alpha}} \left\|\fb-\Pi_{H}\fb\right\|_{L^2(\Omega)}\\ &+& \frac{\sqrt{N_Ed} \big(1+\frac{\mathrm{diam}(\Omega)}{\pi}\big)}{\sqrt{\alpha}} \tilde{\sigma} \frac{\left\|\fb\right\|_{L^2(\Omega)}}{\sqrt{\lambda_{\mathrm{min}}(\mathbf{G})}},
\end{eqnarray}
where $\widetilde{\sigma}$ is given in \eqref{sigma-def}, $N_E$ is the largest number of elements that can possibly be contained within the supporting patches of basis functions, and $\mathbf{G}\in \mathbb{R}^{N_H d\times N_H d}$ is such that \ $\mathbf{G}_{ij}=(\gb_{i},\gb_{j})_{L^2(\Omega)}$, with $\{\gb_i\}_{i\in\{1,\ldots,N_L\}}\subset (\mathbb{Q}^{0}(\mathcal{T}_{H}))^d$ being the basis companion of $\{\varphi_i\}_{i\in\{1,\ldots,N_H\}}$.
\end{theorem}
\begin{proof}
Let $\fb\in L^2(\Omega;\Rd)$ and $\tilde{\ub}=\mathcal{A}^{-1}_{h}\Pi_{H}\fb$. C\'ea's lemma establishes that
\begin{equation*}
\|\ub-\hat{\ub}_H\|_a\leq \|\ub-\hat{\wb}\|_a \;\;\;\;\;\;\text{for all } \hat{\wb}\in \hat{\mathcal{V}}_{H}.
\end{equation*}
Thus, using C\'ea's lemma and the triangle inequality, we obtain
\begin{equation}\label{S-hat-TriangIneq-bound}
\|\ub-\hat{\ub}_H\|_a \leq \|\ub-\ub_h\|_a+\|\ub_h-\tilde{\ub}\|_a+\|\tilde{\ub}-\hat{\wb}\|_a,
\end{equation}
where $\hat{\wb}\in \hat{\mathcal{V}}_{H}$ is arbitrary. To obtain the error estimate, we derive bounds for the last two terms on the r.h.s.\ of the above inequality. 

From \eqref{weak_form}, the definition of $\tilde{\ub}$, the Cauchy-Schwarz inequality, \eqref{A-alpha-beta-bounds}, the property $\|\vb-\Pi_{H}\vb\|_{L^2(\Omega)}\leq \pi^{-1}H \|\nabla \vb\|_{L^2(\Omega)}$ for all $\vb\in H^1(\Omega;\Rd)$, and noticing that $\ub_h-\tilde{\ub}\in \Vh_h \subset H^{1}_{0}(\Omega;\Rd)$, we have
\begin{eqnarray}\label{S-hat-error-bound-rhs1}
\nonumber \|\ub_h-\tilde{\ub}\|^{2}_{a}=a(\ub_h-\tilde{\ub},\ub_h-\tilde{\ub})&=&a(\ub_h,\ub_h-\tilde{\ub})-a(\tilde{\ub},\ub_h-\tilde{\ub})\\
\nonumber &=&(\fb,\ub_h-\tilde{\ub})_{L^2(\Omega)}-(\Pi_{H}\fb,\ub_h-\tilde{\ub})_{L^2(\Omega)}\\
\nonumber &=&(\fb-\Pi_{H}\fb,\ub_h-\tilde{\ub})_{L^2(\Omega)}\\
\nonumber &=&(\fb-\Pi_{H}\fb,\ub_h-\tilde{\ub}-\Pi_{H}(\ub_h-\tilde{\ub}))_{L^2(\Omega)}\\
\nonumber &\leq & \|\fb-\Pi_{H}\fb\|_{L^2(\Omega)}\|\ub_h-\tilde{\ub}-\Pi_{H}(\ub_h-\tilde{\ub})\|_{L^2(\Omega)}\\
&\leq & \frac{H}{\pi \sqrt{\alpha}} \left\|\fb-\Pi_{H}\fb\right\|_{L^2(\Omega)}\|\ub_h-\tilde{\ub}\|_{a}.
\end{eqnarray}
Let $\Pi_{H}\fb=\sum_{i=1}^{N_H}c_i \gb_i$. From the definition of $\tilde{\ub}$, and noticing that $\varphi_i=\mathcal{A}_{h}^{-1}\gb_i$ (since $\gb_i$ is the $\mathbb{Q}^0$-companion of $\varphi_i$), we obtain $\tilde{\ub}=\sum_{i=1}^{N_H}c_i \varphi_i$. With $\hat{\wb}=\sum_{i=1}^{N_H}c_i \hat{\varphi}_i$, $\omega_i=\mathrm{supp}(\hat{\varphi}_i)$, the Cauchy-Schwarz inequality, \Cref{Lemma-Localization-Error-bound-conormalDer}, \eqref{sigma-def}, Lemma~\ref{Riesz-basis}, and since $\|\Pi_{H}\fb\|_{L^2(\Omega)}\leq \|\fb\|_{L^2(\Omega)}$, we can bound the third r.h.s.\ term of \eqref{S-hat-TriangIneq-bound} as follows:
\begin{eqnarray}\label{S-hat-error-bound-rhs2}
\nonumber \|\tilde{\ub}-\hat{\wb}\|_{a}^2&=& \sum_{i=1}^{N_H}c_i a\left(\varphi_{i}-\hat{\varphi}_{i},\tilde{\ub}-\hat{\wb}\right) \\
\nonumber &\leq & \left(\sum_{i=1}^{N_H}c_i^2\right)^{\frac{1}{2}} \left(\sum_{i=1}^{N_H}a\left(\varphi_{i}-\hat{\varphi}_{i},\tilde{\ub}-\hat{\wb}\right)^2 \right)^{\frac{1}{2}} \\
\nonumber &\leq & \frac{\left(1+\frac{\mathrm{diam}(\Omega)}{\pi}\right)}{\sqrt{\alpha}} \tilde{\sigma} \left(\sum_{i=1}^{N_H}\|(\tilde{\ub}-\hat{\wb})_{|_{\omega_i}}\|^2_{a_{\omega_i}}\right)^{\frac{1}{2}}\left(\sum_{i=1}^{N_H}c_i^2\right)^{\frac{1}{2}}\\
&\leq & \frac{\left(1+\frac{\mathrm{diam}(\Omega)}{\pi}\right)}{\sqrt{\alpha}}  \tilde{\sigma}  \left(N_Ed\|(\tilde{\ub}-\hat{\wb}\|^2_{a}\right)^{\frac{1}{2}}\left(\frac{\left\|\fb\right\|_{L^2(\Omega)}}{\sqrt{\lambda_{\mathrm{min}}(\mathbf{G})}}\right),
\end{eqnarray}
where in the last inequality we used the fact that each element of $\mathcal{T}_H$ is contained in the support of $N_E d$ basis functions.

Thus, from \eqref{S-hat-TriangIneq-bound}, \eqref{S-hat-error-bound-rhs1}, and \eqref{S-hat-error-bound-rhs2}, the estimate in (\eqref{S-hat-Error-Estimate}) follows.
\end{proof}

\begin{remark}
The third term of the r.h.s.\ of \eqref{S-hat-Error-Estimate} suggests that the error of the approximate solution due to the basis localization procedure will be small provided that $\widetilde{\sigma}$ is small and $\lambda_{\mathrm{min}}(\mathbf{G})$ is large enough. Moreover, this error due to localization will exhibit a superexponential decay if $\frac{\widetilde{\sigma}}{\sqrt{\lambda_{\mathrm{min}}(\mathbf{G})}}$ does so.
\end{remark}

\subsection{Condition number of SLOD stiffness matrix}
~~The stiffness matrix associated with the SLOD basis is defined as $\hat{\mathbb{A}}_H \in \mathbb{R}^{N_H d\times N_H d}$ such that $(\hat{\mathbb{A}}_H)_{ij}=a(\hat{\varphi}_i,\hat{\varphi}_j)$, where $\hat{\varphi}_i,\hat{\varphi}_j\in \hat{\mathcal{B}}$. Consider $\Pi_{H}\hat{\varphi}_{i}=\sum_{k=1}^d\sum_{T\in \mathcal{T}_{H}}p_{T,k}^{(i)}\eb_k\chi_{T}$ and define $\mathbf{P}\in \mathbb{R}^{N_{H}d \times N_{H}d}$ such that $\mathbf{P}_{ij}=p^{(j)}_{T_q,k}$ with $i=(k-1)N_H+q$. We can write $\mathbf{P}=\widetilde{\mathbf{P}}\mathbf{N}$, where $\widetilde{\mathbf{P}}\in \mathbb{R}^{N_{H}d\times N_{H}d }$ is such that $\widetilde{\mathbf{P}}_{ij}=\|\hat{\varphi}_{j}\|_a\mathbf{P}_{ij}$ and $\mathbf{N}\in\mathbb{R}^{N_{H}d\times N_{H}d }$ is the diagonal matrix such that $\mathbf{N}_{ii}=1/\|\hat{\varphi}_{i}\|_a$. With the same procedure as in \cite[Appendix B]{GarayMohrPeterseimetal}, it can be shown that 
\begin{equation}\label{PTP-lambda-min-bound}
\lambda_{\mathrm{min}}^{-1}(\mathbf{P}^T\mathbf{P})\leq \frac{C}{\lambda_{\mathrm{min}}(\widetilde{\mathbf{P}}^T \widetilde{\mathbf{P}})}H^{d-2},
\end{equation}
where $C$ and $\lambda_{\mathrm{min}}(\widetilde{\mathbf{P}}^T \widetilde{\mathbf{P}})$ are mesh independent quantities, $C$ depends on $\beta$ and $\alpha$, and $\lambda_{\mathrm{min}}(\widetilde{\mathbf{P}}^T \widetilde{\mathbf{P}})$ is independent of $\beta$ and $\alpha$ but depends on the degree of linear independency of the basis functions. With $\hat{\mathbf{P}}=H^{\frac{d}{2}-1}\mathbf{P}$, it follows from \eqref{PTP-lambda-min-bound} that $\lambda_{\mathrm{min}}(\hat{\mathbf{P}}^T \hat{\mathbf{P}})$ has a mesh independent lower bound, i.e., $\lambda^{-1}_{\mathrm{min}}(\hat{\mathbf{P}}^T \hat{\mathbf{P}})=\mathcal{O}(1)$. The following theorem provides an estimate on the condition number of the stiffness matrix $\hat{\mathbb{A}}_H$.

\begin{theorem}\label{HSLOD-Stiffness-Blocks-condnum-estimate}
The condition number of the stiffness matrix $\hat{\mathbb{A}}_{H}$ is $\mathcal{O}(H^{-2})$ and can be bounded as
\begin{equation}\label{SLOD-condnum-bound}
\kappa\big(\hat{\mathbb{A}}_{H} \big)\leq 
\frac{\mathrm{diam}^2(\Omega) \mathfrak{n}_{o}}{\pi^2 \alpha  \lambda_{\mathrm{min}}(\hat{\mathbf{P}}^T\hat{\mathbf{P}})}H^{-2},
\end{equation}
where $\mathfrak{n}_{o}$ is the maximum possible number of functions $\hat{\varphi}_{i}\in \hat{\mathcal{B}}$  whose supports overlap over a region of $\Omega$, $\alpha>0$ is given in \eqref{A-alpha-beta-bounds}, and $\hat{\mathbf{P}}\in \mathbb{R}^{N_H d\times N_H d}$ is defined as above with $\lambda^{-1}_{\mathrm{min}}(\hat{\mathbf{P}}^T \hat{\mathbf{P}})=\mathcal{O}(1)$. 
\end{theorem}
\begin{proof}
Using \Cref{Riesz-basis} and for an arbitrary $\mathbf{c}=\left(c_i\right)_{i=1}^{N_{H}}\in \mathbb{R}^{N_{H}}$ we have
\begin{equation}\label{L2norm-lower-bound-lincomb-L2projections}
\Big\|\sum_{i=1}^{N_{H}}c_i\Pi_{H}\hat{\varphi}_{i}\Big\|_{L^2(\Omega)}^2=H^d\mathbf{c}^T\mathbf{P}^T\mathbf{P} \mathbf{c}\geq H^d \lambda_{\mathrm{min}}(\mathbf{P}^T\mathbf{P})\|\mathbf{c}\|_2^2.
\end{equation}
Since the $L^2$-projection operator $\Pi_{H}$ is an orthogonal projection operator, it holds that
\begin{equation}\label{L2-proj-reduced-norm-condition}
\|\Pi_{H}\vb\|_{L^2(\Omega)}\leq\|\vb\|_{L^2(\Omega)}\quad\text{for all }\vb\in L^2(\Omega;\Rd).
\end{equation}
Then, from \eqref{L2norm-lower-bound-lincomb-L2projections}, \eqref{L2-proj-reduced-norm-condition}, the Poincar\'e-Friedrichs inequality, and \eqref{A-alpha-beta-bounds}, we obtain
\begin{equation}\label{condnum-derivation}
 H^d\lambda_{\mathrm{min}}(\mathbf{P}^T\mathbf{P})\sum_{i=1}^{N_{H}}c_{i}^{2}\leq  \Big\|\sum_{i=1}^{N_{H}}c_i\Pi_{H}\hat{\varphi}_{i}\Big\|_{L^2(\Omega)}^2
\leq \Big\|\sum_{i=1}^{N_{H}}c_i\hat{\varphi}_{i}\Big\|_{L^2(\Omega)}^2\leq \frac{\mathrm{diam}^2(\Omega)}{\pi^2 \alpha} \Big\|\sum_{i=1}^{N_{H}}c_i\hat{\varphi}_{i}\Big\|_{a}^2.
\end{equation}
From \Cref{Riesz-basis}, the definition of $\hat{\mathbf{P}}$, and \eqref{condnum-derivation} it follows that
\begin{equation}
\lambda_{\mathrm{min}}(\hat{\mathbb{A}}_H)\geq \frac{\pi^2\alpha}{\mathrm{diam}^2(\Omega)}H^{2}\lambda_{\mathrm{min}}(\hat{\mathbf{P}}^T\hat{\mathbf{P}}).
\end{equation}

Since $|a(\hat{\varphi}_{i},\hat{\varphi}_{j})|\leq 1$ for all $\hat{\varphi}_{i},\hat{\varphi}_{j}\in \hat{\mathcal{B}}$, using the Gershgorin Circle Theorem we obtain $\lambda_{\mathrm{max}}(\hat{\mathbb{A}}_H)\leq \mathfrak{n}_{o}$. Then, since $\kappa(\hat{\mathbb{A}}_H)=\lambda_{\max}(\hat{\mathbb{A}}_H)/\lambda_{\min}(\hat{\mathbb{A}}_H)$, the estimate \eqref{SLOD-condnum-bound} follows.
\end{proof}

\begin{remark}
Note from \eqref{PTP-lambda-min-bound} that $\lambda^{-1}_{\mathrm{min}}(\hat{\mathbf{P}}^T\hat{\mathbf{P}})\leq C \lambda^{-1}_{\mathrm{min}}({\widetilde{\mathbf{P}}^T \widetilde{\mathbf{P}}})$. From the definition of $\widetilde{\mathbf{P}}$ and the definition of SLOD basis functions given \eqref{SLOD basis def}, it follows that $\widetilde{\mathbf{P}}=\mathbf{I}+\mathbf{C}$, where $\mathbf{I}$ is the identity matrix and $\mathbf{C}\in\mathbb{R}^{n_e d \times n_e d}$ is the matrix whose $i$-th row contains the vector $\bar{\mathbf{c}}_s$ (cf.\ \eqref{stable_c_computation}) associated with the SLOD basis function $\hat{\varphi}_i$ and a zero diagonal entry. If $\mathbf{C}=0$ we would have $\lambda_{\mathrm{min}}({\widetilde{\mathbf{P}}^T \widetilde{\mathbf{P}}})=1$. Consequently, whenever the entries of $\mathbf{C}$ are small enough (which is the goal of the basis stability condition \eqref{basis-stability-condition}), $\lambda_{\mathrm{min}}({\widetilde{\mathbf{P}}^T \widetilde{\mathbf{P}}})$ should be far from zero and $\lambda^{-1}_{\mathrm{min}}({\widetilde{\mathbf{P}}^T \widetilde{\mathbf{P}}})$ small.
\end{remark}

\section{Numerical Experiments}\label{s:numexp}

All numerical simulations have been carried out using the \texttt{deal.II} library \cite{deal}, 
and the source code is publicly available at 
\url{https://github.com/camillabelponer/dealii-SLOD}.

The code developed for this work follows an object-oriented design using class inheritance. 
A virtual base class \texttt{LOD} encapsulates the core, problem-independent aspects of 
the Localized Orthogonal Decomposition (LOD) method, including problem setup and solution 
assembly. Depending on user-specified parameters from an input file, the code can construct 
either LOD basis functions 
\(\bar{\psi}^{\text{LOD}}_{T,k}\) 
or SLOD basis functions 
\(\hat{\varphi}^{\text{SLOD}}_{T,k}\), 
which are then used to assemble the system matrix.

To handle specific problem settings, a (publicly) derived class, \texttt{ElasticityProblem}, 
specializes \texttt{LOD} for linear elasticity. It contains problem-specific elements such 
as the elasticity parameters and the assembly routine for the fine-scale stiffness matrix  associated with
\(\mathcal{A}_{h,\omega}\). Another derived class, \texttt{DiffusionProblem}, is also
implemented and allowed us to compare directly this implementation with literature (\cite{GarayMohrPeterseimetal}), however results for Diffusion problems are nor discussed in this work.

Unlike previous prototype SLOD implementations (see, e.g., 
\cite{haucklozinsky, MalP20} 
), the assembly of the SLOD system matrix 
in this code does not rely on the global fine-scale FEM system matrix. However, the global fine-scale 
matrix can be optionally assembled if a reference solution is needed for error 
computation. This is accomplished by storing an additional vector for each basis function 
on each patch, which decouples the SLOD system matrix assembly from the global standard FEM matrix.

Algorithm~\ref{alg:cap} summarizes the SLOD basis construction step, producing two basis functions 
for each patch, as detailed in Section~\ref{section:SLOD}.

\begin{algorithm}
  \caption{Construction of SLOD basis functions}\label{alg:cap}
  \begin{algorithmic}
    \For{each patch $\omega$ centerd in $T$}
      \State \textbf{Assemble} \texttt{patch\_stiffness\_matrix} associated with $\mathcal{A}_{h,\omega}$ (cf. \eqref{LOD-Saddle-point-formulation})
      \State \textbf{Assemble} \texttt{projection\_matrix} associated with $\mathcal{P}^T$ (cf. \eqref{LOD-Saddle-point-formulation})
      \State \textbf{Assemble} matrices $\mathbf{B}, \mathbf{D}$ \quad({cf.}\ \eqref{linear-system-solved-by-least-squares})
      \State \textbf{Compute} singular value decomposition of $(\mathbf{BD})^T \mathbf{BD}$ 
             \quad({cf.}\ \eqref{svd-computation}) \Comment{Uses LAPACK}
      \State \textbf{Compute} correction $\bar{\mathbf{c}}_{s}$ \quad({cf.}\ \eqref{stable_c_computation})
      \State \textbf{Compute} basis functions 
             $\hat{\varphi}^{\text{SLOD}}_{T,i}, i\in\{1,\ldots,d\}$
    \EndFor
  \end{algorithmic}
\end{algorithm}

In what follows, we present the results of two numerical experiments with different data. The first considers constant Lam\'e parameters and a constant right-hand side $\fb$. The second assumes strongly heterogeneous Lam\'e parameters and two types of $\fb$. In these experiments, errors are measured in the $H^1$-seminorm and $L^2$-norm. Note that the decay rate of the error in the $H^1$-seminorm is the same as in the energy norm due to the equivalence between the energy norm and the $H^1$-seminorm indicated in \eqref{Equivalence-Energy-H1semi}. From standard duality arguments, the error in the $L^2$-norm is expected to have an extra order of decay in $H$ with respect to the energy norm, as is known for the classical LOD method in the scalar case; see e.g. \cite[Ch.~5]{MalP20}.  

\subsection{Simple convergence test}
We now present a set of numerical tests to compare the SLOD approach 
against the original LOD method and the standard FEM. 
In this experiment, the problem domain \(\Omega\) is the unit square $(0,1)^2$, with homogeneous 
Dirichlet boundary data on \(\partial \Omega\). The right-hand side is 
\(\fb = (1, 1)^{\top}\), and the Lam\'{e} parameters are set to \(\lambda = \mu = 1\). 
Errors are computed against a fine-scale FEM reference solution \(\ub_h\), obtained 
on a fine mesh with \(h = 2^{-8}\). 

Figure~\ref{fig:test_convergence} shows the error behavior of SLOD compared to the classical 
FEM. The oversampling parameter \(m\) controls how much additional local information is 
incorporated in the basis functions. With \(m \geq 2\) a very coarse 
mesh is sufficient to reach the same accuracy that standard FEM attains at its finest level 
of refinement. This demonstrates that the super-localized basis effectively captures 
fine-scale features, leading to significant computational savings.

Figure~\ref{fig:test_convergence_m} illustrates the superexponential decay of the SLOD error 
for a range of coarse-mesh sizes~\(H\). This finding aligns with earlier results for scalar 
problems that employ a similar basis stabilization technique \cite{GarayMohrPeterseimetal} on scalar-valued problems, 
thus validating our current implementation.

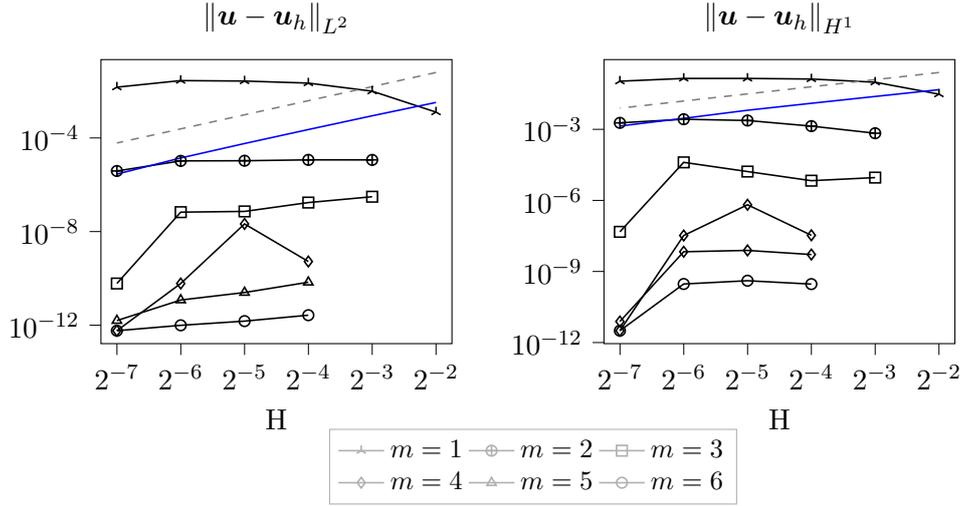
\begin{figure}[ht]
  \centering
\begin{tikzpicture}

\begin{axis}[
      name=ax1,
      ErrorPlot,
      xmin=0.00656950324416964, xmax=0.29730177875068,
      xmode=log,
      title={$\left\lVert \boldsymbol{u} - \boldsymbol{u}_h \right\lVert _{L^2}$},
      ymin=1.63357852714323e-13, ymax=0.222555875924612,
      ymode=log
]


\addplot [p2]
table {%
0.25 0.001271
0.125 0.009942
0.0625 0.02204
0.03125 0.02696
0.015625 0.02819
0.0078125 0.01475
};
\addlegendentry{
$m = 1$}
\addplot [p3]
table {%
0.125 1.154e-05
0.0625 1.159e-05
0.03125 1.064e-05
0.015625 1.047e-05
0.0078125 3.88e-06
};
\addlegendentry{
$m = 2$}
\addplot [p4]
table {%
0.125 3.081e-07
0.0625 1.737e-07
0.03125 7.27e-08
0.015625 6.759e-08
0.0078125 6.051e-11
};
\addlegendentry{
$m = 3$}
\addplot [p5]
table {%
0.0625 5.262e-10
0.03125 2.123e-08
0.015625 6.048e-11
0.0078125 5.817e-13
};
\addlegendentry{
$m = 4$}
\addplot [p6]
table {%
0.0625 6.862e-11
0.03125 2.495e-11
0.015625 1.192e-11
0.0078125 1.577e-12
};
\addlegendentry{
$m = 5$}
\addplot [p7]
table {%
0.0625 2.679e-12
0.03125 1.495e-12
0.015625 9.891e-13
0.0078125 5.817e-13
};
\addlegendentry{
$m = 6$}
\addplot [fem]
table {%
0.25 0.003287
0.125 0.0008842
0.0625 0.0002288
0.03125 5.783e-05
0.015625 1.397e-05
0.0078125 2.854e-06
};
\addplot [order]
table {%
0.25 0.0625
0.125 0.015625
0.0625 0.00390625
0.03125 0.0009765625
0.015625 0.000244140625
0.0078125 6.103515625e-05
};
\end{axis}
\begin{axis}[
  at={(ax1.south east)},
  xshift=2cm,
  ErrorPlot,
  xmin=0.00656950324416964, xmax=0.29730177875068,
xmode=log,
title={$\left\lVert \boldsymbol{u} - \boldsymbol{u}_h \right\lVert _{H^1}$},
ymin=8.99605975117506e-13, ymax=0.876772744753029,
ymode=log
]
\addplot [p2]
table {%
0.25 0.03055
0.125 0.09721
0.0625 0.1326
0.03125 0.1406
0.015625 0.1405
0.0078125 0.1072
};
\addplot [p3]
table {%
0.125 0.0006832
0.0625 0.001367
0.03125 0.002356
0.015625 0.002665
0.0078125 0.001875
};
\addplot [p4]
table {%
0.125 9.177e-06
0.0625 6.787e-06
0.03125 1.65e-05
0.015625 4.036e-05
0.0078125 4.712e-08
};
\addplot [p5]
table {%
0.0625 3.366e-08
0.03125 6.538e-07
0.015625 3.293e-08
0.0078125 3.155e-12
};
\addplot [p5]
table {%
0.0625 5.182e-09
0.03125 7.743e-09
0.015625 6.797e-09
0.0078125 7.722e-12
};
\addplot [p7]
table {%
0.0625 2.93e-10
0.03125 4.068e-10
0.015625 2.96e-10
0.0078125 3.155e-12
};
\addplot [fem]
table {%
0.25 0.04706
0.125 0.02439
0.0625 0.01255
0.03125 0.006381
0.015625 0.002908
0.0078125 0.001377
};
\addplot [order]
table {%
0.25 0.25
0.125 0.125
0.0625 0.0625
0.03125 0.03125
0.015625 0.015625
0.0078125 0.0078125
};

\end{axis}

\end{tikzpicture}
  \caption{Error of SLOD compared to the classical FEM (blue line). 
           The dashed gray line indicates the theoretical FEM convergence rate, 
           $\mathcal{O}(H^2)$ (left) and $\mathcal{O}(H)$ (right).}
  \label{fig:test_convergence}
\end{figure}

\begin{figure}[ht]
  \centering
  \begin{tikzpicture}
\begin{axis}[
  ErrorPlot,
legend cell align={left},
legend style={
  fill opacity=0.8,
  draw opacity=1,
  text opacity=1,
  at={(1.05,0.05)},
  anchor=south west,
  draw=darkgray176,
  legend columns=1
},
xlabel={m},
xmin=0.85, xmax=6.15,
xtick style={color=black},
xtick={1, 2, 3, 4, 5, 6},
title={$\left\lVert \boldsymbol{u} - \boldsymbol{u}_h \right\lVert _{L^2}$},
ymin=1e-13, ymax=0.1,
ymode=log,
]
\addplot [p3, dashed]
table {%
1 0.01888
2 0.002767
3 0.0006385
};
\addlegendentry{$H= 2^{-3}$}
\addplot [p4, dashed]
table {%
1 0.02696
2 0.01207
3 0.001357
4 0.0003363
5 0.0003012
6 0.0002503
};
\addlegendentry{$H= 2^{-4}$}
\addplot [p5, dashed]
table {%
1 0.02945
2 0.02281
3 0.005453
4 0.0005145
5 6.37e-05
6 1.554e-05
};
\addlegendentry{$H= 2^{-5}$}
\addplot [p6, dashed]
table {%
1 0.0301
2 0.02815
3 0.01544
4 0.002492
5 0.0002297
6 2.144e-05
};
\addlegendentry{$H= 2^{-6}$}
\addplot [p7, dashed]
table {%
1 0.03028
2 0.0301
3 0.0285
4 0.01968
5 0.005141
6 0.0006752
};
\addlegendentry{$H= 2^{-7}$}
\addplot [p3, forget plot]
table {%
1 0.009942
2 1.154e-05
3 3.081e-07
};
\addplot [p4, forget plot]
table {%
1 0.02204
2 1.159e-05
3 1.737e-07
4 5.262e-10
5 6.862e-11
6 2.679e-12
};
\addplot [p5, forget plot]
table {%
1 0.02696
2 1.064e-05
3 7.27e-08
4 2.123e-08
5 2.495e-11
6 1.495e-12
};
\addplot [p6, forget plot]
table {%
1 0.02819
2 1.047e-05
3 6.759e-08
4 6.048e-11
5 1.192e-11
6 9.891e-13
};
\addplot [p7, forget plot]
table {%
1 0.01475
2 3.88e-06
3 6.051e-11
4 5.817e-13
5 1.577e-12
6 5.817e-13
};
\addplot[order, samples=1000, domain=1:6] {1e-4*exp(-x*x)}; 
\addlegendentry{$ \propto e^{-m^2}$}
\addplot[semithick, gray, samples=1000, domain=1.5:6] {0.015*exp(-1.4 * x)};
\addlegendentry{$ \propto e^{-m}$}
\end{axis}

\end{tikzpicture}
  \caption{Exponential decay of SLOD (solid lines). Dashed lines show the LOD error 
           for the same values of $H$ given in the legend.}
  \label{fig:test_convergence_m}
\end{figure}
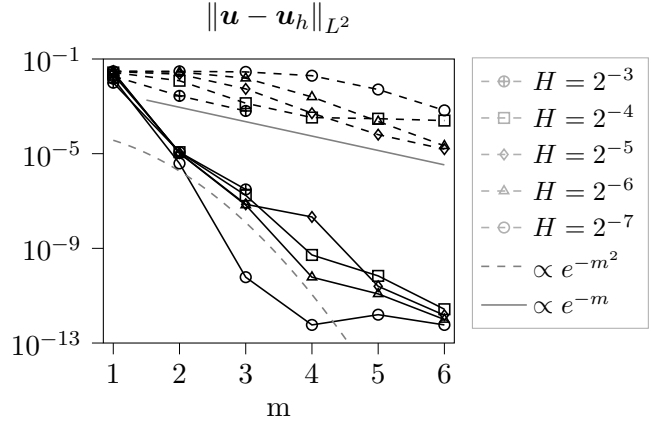

The purpose of this ideal test is twofold: 
(1)~to confirm our implementation, and 
(2)~to investigate the role of the method parameters \((m, H)\) in identifying 
an effective setup for practical use. Based on these experiments, we hereafter restrict 
attention to $m>1$ but $m<5$. Figure~\ref{fig:test_convergence} is consistent with the
existing literature in showing that a single oversampling layer ($m=1$) is insufficient
to fully exploit the efficiency of SLOD, while $m=5$ or $m=6$ offers no significant
accuracy gain relative to its higher computational cost. In other words, an oversampling
level between $2$ and $4$ appears to strike a more favorable balance between accuracy and computational effort.

\subsection{Strongly heterogeneous coefficient}

We now turn to a more practical test case. First, we create a grid of characteristic length 
\(\eta = 2^{-6}\), on which the Lam\'{e} parameters \(\lambda\) and \(\mu\) are sampled 
independently from a uniform distribution over the interval \([1,100]\). 
We then investigate how these heterogeneous coefficients impact the performance 
of the SLOD method.

\begin{figure}[ht]
  \centering
  \includegraphics[height=0.25\textwidth,keepaspectratio]{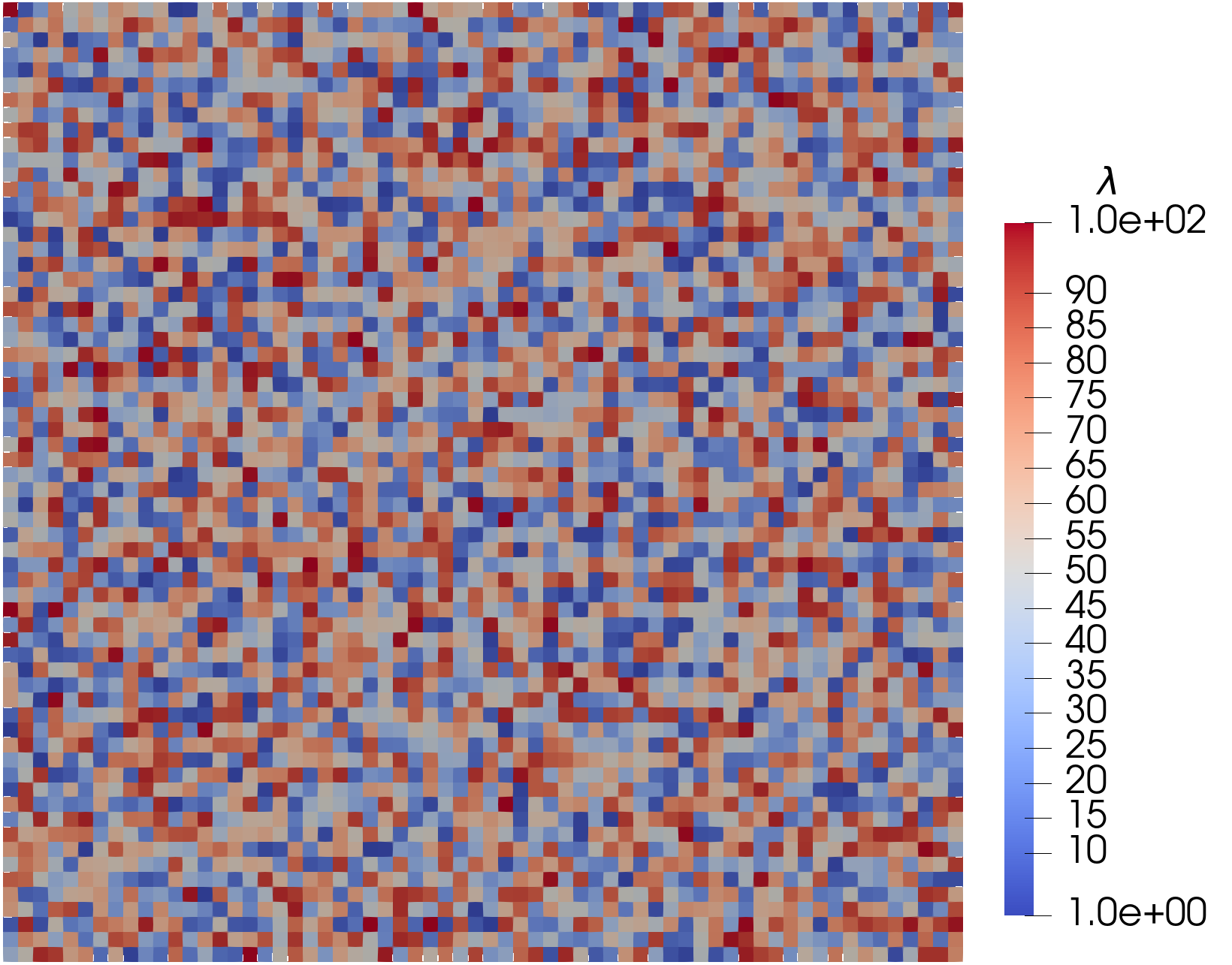}
  \hspace{0.6cm}
  \includegraphics[height=0.25\textwidth,keepaspectratio]{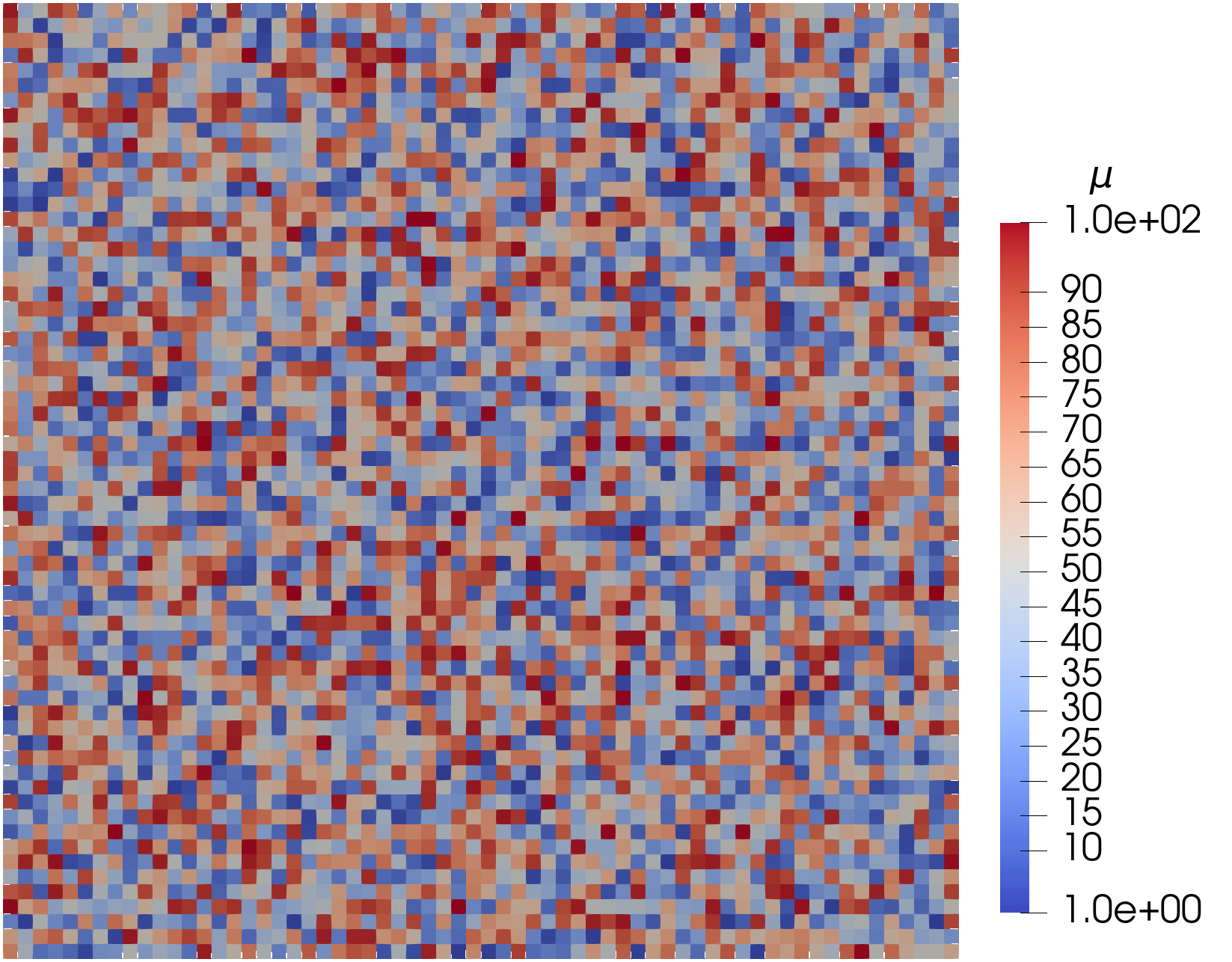}
  \caption{Realizations of \(\lambda\) and \(\mu\) drawn from \(\mathcal{U}[1, 100]\) on 
  a grid of characteristic length \(\eta = 2^{-6}\).}
\end{figure}

A fine-scale FEM reference solution \(\ub_h\) is computed on a mesh with \(h = 2^{-8}\). 
We then compare this reference to approximate solutions obtained with both SLOD and a (coarse) FEM discretizations 
on meshes of increasing refinement \(H = 2^{-3},\dots,2^{-7}\). 
The right-hand side is chosen as the constant function \(\fb = (1,1)\). 
Figure~\ref{fig:test_randomcoeff_pwc} shows that SLOD exhibits a superexponential 
decay in the error, whereas the (coarse) FEM recovers its optimal rate of convergence once the mesh size $H$ approaches characteristic length $\eta$.

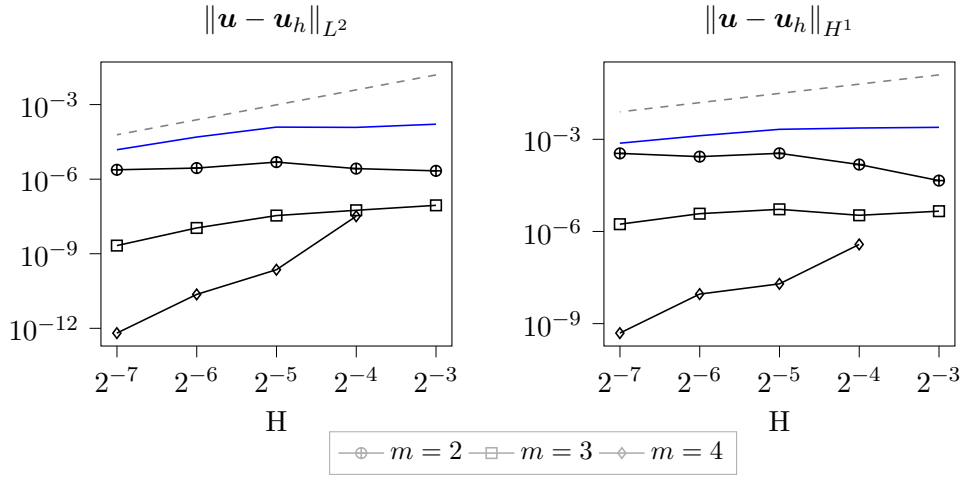
\begin{figure}[!ht]
  \centering

\begin{tikzpicture}
  \begin{axis}[
    name=ax1,
  ErrorPlot,
xmin=0.00680117627575097, xmax=0.143587294374629,
xmode=log,
title={$\left\lVert \boldsymbol{u} - \boldsymbol{u}_h \right\lVert _{L^2}$},
ymin=1.94060243391216e-13, ymax=0.0516592158435566,
ymode=log,
]
\addplot [p3]
table {%
0.125 2.161e-06
0.0625 2.658e-06
0.03125 4.854e-06
0.015625 2.791e-06
0.0078125 2.385e-06
};
\addlegendentry{$m = 2$};
\addplot [p4]
table {%
0.125 8.852e-08
0.0625 5.602e-08
0.03125 3.431e-08
0.015625 1.089e-08
0.0078125 2.126e-09
};
\addlegendentry{$m = 3$};
\addplot [p5]
table {%
0.0625 3.234e-08
0.03125 2.275e-10
0.015625 2.31e-11
0.0078125 6.416e-13
};
\addlegendentry{$m = 4$};
\addplot [fem]
table {%
0.125 0.0001624
0.0625 0.0001209
0.03125 0.000124
0.015625 4.947e-05
0.0078125 1.524e-05
};
\addplot [order]
table {%
0.125 0.015625
0.0625 0.00390625
0.03125 0.0009765625
0.015625 0.000244140625
0.0078125 6.103515625e-05
};
\end{axis}
\begin{axis}[
  at={(ax1.south east)},
  xshift=2cm,
  ErrorPlot,
  xmin=0.00680117627575097, xmax=0.143587294374629,
xmode=log,
title={$\left\lVert \boldsymbol{u} - \boldsymbol{u}_h \right\lVert _{H^1}$},
ymin=1.86627558852101e-10, ymax=0.32899749842765,
ymode=log
]
\addplot [p3]
table {%
0.125 4.532e-05
0.0625 0.0001508
0.03125 0.0003498
0.015625 0.0002728
0.0078125 0.0003478
};
\addplot [p4]
table {%
0.125 4.573e-06
0.0625 3.359e-06
0.03125 5.26e-06
0.015625 3.791e-06
0.0078125 1.723e-06
};
\addplot [p5]
table {%
0.0625 3.785e-07
0.03125 1.97e-08
0.015625 9.153e-09
0.0078125 4.912e-10
};
\addplot [fem]
table {%
0.125 0.002453
0.0625 0.002334
0.03125 0.002112
0.015625 0.001307
0.0078125 0.0007475
};
\addplot [order]
table {%
0.125 0.125
0.0625 0.0625
0.03125 0.03125
0.015625 0.015625
0.0078125 0.0078125
};
\end{axis}
\end{tikzpicture}
  \caption{Error behavior of SLOD (with oversampling \(m=2,3,4\)) and standard FEM (blue line) 
  in the case of randomly varying Lam\'{e} coefficients and constant $\fb$. 
  The dashed gray lines represent the theoretical FEM convergence orders, 
  \(\mathcal{O}(H^2)\) (left) and \(\mathcal{O}(H)\) (right).}
  \label{fig:test_randomcoeff_pwc}
\end{figure}

A key aspect of the SLOD approximation arises from the upper bound in the error estimate 
\eqref{S-hat-Error-Estimate}, where the term 
\(\left\|\fb-\Pi_{H}\fb\right\|_{L^2(\Omega)}\)
can become dominant when \(\fb\) is smoother. To illustrate this effect, we repeat 
the above test using the smooth right-hand side
\begin{align}\label{eq:test_randomcoeff_reg_rhs}
  \fb \;=\;
  \pi^2
  \begin{bmatrix}
    4\,\sin(2\pi y)\,\bigl(-1 + 2\cos(2\pi x)\bigr)\;-\;\cos\bigl(\pi(x+y)\bigr)\\[6pt]
    4\,\sin(2\pi x)\,\bigl(-1 + 2\cos(2\pi y)\bigr)\;-\;\cos\bigl(\pi(x+y)\bigr)
  \end{bmatrix},
\end{align}
which particularly belongs to \(H^1(\Omega;\mathbb{R}^2)\).

Figure~\ref{fig:test_randomcoeff_reg} shows how, in this scenario, the error is dominated by 
the approximation of the forcing term, which behaves as \(\mathcal{O}(H)\). 
Under the assumption that the localization error remains small, the second term in the 
right-hand side of \eqref{S-hat-Error-Estimate} determines the overall convergence rate. 
Hence we observe \(\mathcal{O}(H^3)\) for the \(L^2\)-error and \(\mathcal{O}(H^2)\) for 
the \(H^1\)-error, consistent with a smooth forcing term in elasticity problems.

\begin{figure}[!ht]
  \centering

\begin{tikzpicture}
\begin{axis}[
    name=ax1,
  ErrorPlot,
  xmin=0.0140820384782942, xmax=0.138696184008481,
xmode=log,
title={$\left\lVert \boldsymbol{u} - \boldsymbol{u}_h \right\lVert _{L^2}$},
ymin=2.53026744574434e-05, ymax=0.0211654121908953,
ymode=log,
]
\addplot [p3]
table {%
0.125 0.0140389
0.0625 0.00785848
0.03125 0.00144822
0.015625 0.000386791
};
\addlegendentry{$m = 2$};
\addplot [p4]
table {%
0.125 0.013985
0.0625 0.00735937
0.03125 0.000864469
0.015625 5.96218e-05
};
\addlegendentry{$m = 3$};
\addplot [p5]
table {%
0.0625 0.00735895
0.03125 0.000865706
0.015625 5.98815e-05
};
\addlegendentry{$m = 4$};
\addplot [fem]
table {%
0.125 0.0102091
0.0625 0.0075931
0.03125 0.00724302
0.015625 0.00313834
};
\addplot[order, samples=1000, domain=0.015625:0.09] (x,10*x*x*x);
\end{axis}
\begin{axis}[
  at={(ax1.south east)},
  xshift=2cm,
   ErrorPlot,
   xmin=0.0140820384782942, xmax=0.138696184008481,
xmode=log,
title={$\left\lVert \boldsymbol{u} - \boldsymbol{u}_h \right\lVert _{H^1}$},
ymin=0.00671915370778831, ymax=0.195305487367985,
ymode=log
]
\addplot [p3]
table {%
0.125 0.129672
0.0625 0.0958878
0.03125 0.0414062
0.015625 0.02359
};
\addplot [p4]
table {%
0.125 0.129326
0.0625 0.0919037
0.03125 0.0300879
0.015625 0.00783128
};
\addplot [p5]
table {%
0.0625 0.0918969
0.03125 0.0301103
0.015625 0.0078448
};
\addplot [fem]
table {%
0.125 0.16757
0.0625 0.154354
0.03125 0.139878
0.015625 0.0854935
};
\addplot[order, samples=1000, domain=0.0275:0.109] (x,10*x*x);
\end{axis}
\end{tikzpicture}
  \caption{Error behavior of SLOD (with oversampling \(m=2,3,4\)) and standard FEM (blue line) 
    for randomly varying Lam\'{e} coefficients and the smooth right-hand side 
    in \eqref{eq:test_randomcoeff_reg_rhs}. 
    Since \(\fb \in H^1(\Omega;\mathbb{R}^2)\) and the localization error is small, 
    the second term in the right-hand side of \eqref{S-hat-Error-Estimate} dominates. 
    The dashed gray lines indicate the theoretical rates: \(\mathcal{O}(H^3)\) for the 
    \(L^2\)-error and \(\mathcal{O}(H^2)\) for the \(H^1\)-error.}
  \label{fig:test_randomcoeff_reg}
\end{figure}
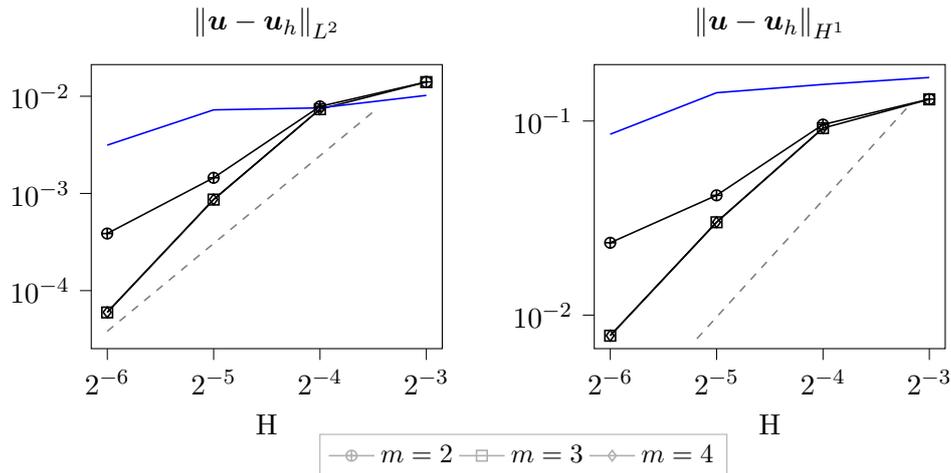

\section{Conclusion}
This paper introduced the Super-Localized Orthogonal Decomposition (SLOD) method for solving 
linear elasticity problems with highly heterogeneous (multiscale) microstructures. By constructing 
super-localized basis functions, SLOD improves upon the standard LOD method in both localization 
and computational efficiency, yet retains comparable theoretical guarantees on accuracy.

Extending the superlocalization framework to vector-valued elasticity underscores SLOD's 
adaptability for diverse engineering and physical applications. Our numerical analysis 
demonstrates stability and convergence, instilling confidence in the method’s reliability. 
Meanwhile, the \texttt{deal.II}-based implementation showcases SLOD’s scalability and 
integration potential in modern high-performance computing workflows, an important step for 
addressing large-scale multiscale problems.

Several directions for future research remain open. First, while numerical evidence supports 
the super-exponential decay of the localization error, deeper theoretical proofs, particularly 
regarding the spectral geometry conjectures in the scalar setting, would further solidify 
SLOD’s foundations. Second, given the success of standard LOD for thermo- and poro-elasticity 
in \cite{MaPersson17, FuACMPP19, AltCMPP20}, extending superlocalization to broader 
multiphysics problems where elasticity couples with other physical processes could yield 
significant new capabilities. Finally, enhancing our parallel implementation and incorporating 
adaptive refinements may expand SLOD’s applicability to even more demanding simulations.

\bibliographystyle{plainnat}
\bibliography{bibliography.bib}

\end{document}